\documentclass[final,onefignum,onetabnum]{siamart190516}

\usepackage{lipsum}
\usepackage{amsfonts}
\usepackage{graphicx}
\usepackage{epstopdf}
\usepackage[algo2e,linesnumbered,vlined]{algorithm2e}
\usepackage{amssymb}
\usepackage{subfig}
\usepackage{adjustbox}

\usepackage{tikz,pgfplots,pgfplotstable}
\pgfplotsset{compat=1.16}
\usetikzlibrary{arrows}

\usepackage{array}
\newcolumntype{P}[1]{>{\centering\arraybackslash}p{#1}}
\allowdisplaybreaks

\ifpdf
  \DeclareGraphicsExtensions{.eps,.pdf,.png,.jpg}
\else
  \DeclareGraphicsExtensions{.eps}
\fi

\newsiamremark{remark}{Remark}
\newsiamremark{hypothesis}{Hypothesis}
\newsiamremark{example}{Example}
\crefname{hypothesis}{Hypothesis}{Hypotheses}
\newsiamthm{claim}{Claim}

\makeatletter
\renewcommand{\ALG@name}{\sc Algorithm}
\makeatother

\headers{Learning Symbolic Expressions}{J. Kim, S. Leyffer, and P. Balaprakash}

\title{Learning Symbolic Expressions: Mixed-Integer Formulations, Cuts, and Heuristics\footnote{Argonne National Laboratory, Preprint ANL/MCS-P9445-0212.}}

\author{Jongeun Kim\thanks{Industrial and Systems Engineering, University of Minnesota Twin Cities, Minneapolis, MN 55455, USA
  (\email{kim00623@umn.edu}).}
\and Sven Leyffer\thanks{Mathematics and Computer Science Division, Argonne National Laboratory, Lemont, IL 60439, USA 
  (\email{leyffer@mcs.anl.gov}, \email{pbalapra@anl.gov}).}
\and Prasanna Balaprakash\footnotemark[2]}

\usepackage{amsopn}

\newcommand{\R}{\mbox{I}\!\mbox{R}}

\newcommand{\be}{\begin{equation}}
\newcommand{\ee}{\end{equation}}
\newcommand{\bea}{\begin{eqnarray}}
\newcommand{\eea}{\end{eqnarray}}

\newcommand{\bvec}{\left(\begin{array}{c}}
\newcommand{\evec}{\end{array}\right)}
\newcommand{\bsub}{\begin{subequations}}
\newcommand{\esub}{\end{subequations}}

\newcommand{\mc}[3]{\multicolumn{#1}{#2}{#3}}

\mathchardef\mhyphen="2D

\begin{document}

\maketitle

\begin{abstract}
In this paper we consider the problem of learning a regression function without assuming its functional form.
This problem is referred to as  symbolic regression.
An expression tree is typically used to represent a solution function, which is determined by assigning operators and operands to the nodes.
The symbolic regression problem can be formulated as a nonconvex mixed-integer nonlinear program (MINLP),
where binary variables are used to assign operators and nonlinear expressions are used to propagate 
data values through nonlinear operators such as square, square root, and exponential.
We extend this formulation by adding new cuts that improve the solution of this challenging MINLP.
We also propose a heuristic that iteratively builds an expression tree by solving a restricted MINLP.
We perform computational experiments 
and compare our approach with a mixed-integer program-based method and a neural-network-based method from the literature.
\end{abstract}




\section{Introduction}
\label{sec:sr-intro}

We consider the problem of learning symbolic expressions, which is referred to as \emph{symbolic regression}.
Symbolic regression is a form of regression that learns functional expressions from observational data. 
Unlike traditional regression, symbolic regression does not assume a fixed functional form but
instead learns the functional relationship and its constants. Given observational data in terms of
independent variables, $x_i \in \R^d$, and dependent variables (function values), $z_i \in \R$, for
$i=1,\ldots,n_{\text{data}}$,
symbolic regression aims to find the  best functional form that maps the $x$-values to the $z$-values by
solving the following optimization problem:
\begin{equation}
    \label{eq:SymbRegr}
    \min_{f \in {\cal F}} \; \sum_{i=1}^{n_{\text{data}}} \left( f(x_i) - z_i \right)^2,
\end{equation}
where ${\cal F}$ is the space of functions from which $f$ is chosen. We note that other loss functions
involving general norms are also possible and that, in general, problem \cref{eq:SymbRegr} is an infinite-dimensional
optimization problems.
Various applications of symbolic regression have been presented in different fields including 
materials science \cite{wang2019symbolic}, 
fluid systems \cite{duriez2017machine},
physics \cite{Schmidt81,udrescu2020ai2,udrescu2020ai}, 
and civil engineering \cite{tarawneh2019intelligent}.

Symbolic regression is especially useful when we do not know the precise functional form that 
relates the independent variables $x$ to the dependent variables $z$ or when we wish to exploit the
freedom of optimally choosing the functional form.
Given data $(x_i,z_i) \in \R^{d+1}$, $i=1,\ldots,n_{\text{data}}$ 
and a set of mathematical operators, symbolic regression searches for a best-fit mathematical expression as a combination of 
these operators, independent variables, and constant. 
Given suitable restrictions on the function
space ${\cal F}$ (e.g., a finite set of mathematical operators), we can formulate \cref{eq:SymbRegr} as a 
nonconvex mixed-integer nonlinear program (MINLP).
In this paper, we describe new cutting planes to enhance the MINLP formulation, develop new heuristics to
solve the resulting MINLP, and demonstrate the effectiveness of our approach on a broad set of test problems.

The remainder of this paper is organized as follows. 
In the rest of this section, we review some background material on expression trees and the literature on symbolic regression. 
In \cref{sec:sr-formulation}, we introduce a MINLP formulation that improves the existing formulations by adding new sets of cutting planes.
In \cref{sec:sr-heuristic}, we introduce the sequential tree construction heuristic for solving the MINLPs arising in symbolic regression. 
We demonstrate the effectiveness of our ideas in a detailed numerical comparison in \cref{sec:sr-exp}, before concluding with some final remarks in \cref{sec:sr-conclusion}.


\subsection{Review of Expression Trees}

A mathematical expression can be represented by an expression tree.
\Cref{Fig:Expr} shows an expression tree of the pendulum formula (in general, an expression tree is 
not unique).
An expression tree can be constructed by assigning an operand (an independent variable $(x_j)$ or constant (cst)) or an operator ($+$, $-$, $*$, $/$, $\exp$, $\log$, $(\cdot)^2$, $(\cdot)^3$, $\sqrt{}$, etc.) to the nodes on a tree.

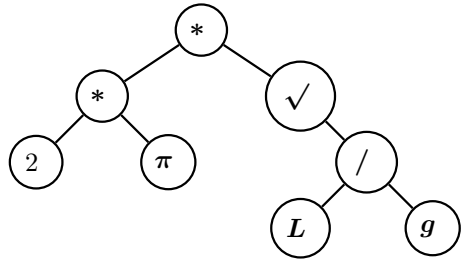
\begin{figure}[htb]
    \centering
    \begin{tikzpicture}
    [scale=0.5,
    every node/.style = {shape=circle, align=center, draw, thick, line width=0.3mm, font={\boldmath}},
    edge from parent/.style = {draw, line width=0.3mm},
    level distance=5em,
    ]
    \tikzstyle{level 1}=[sibling distance=15em]
    \tikzstyle{level 2}=[sibling distance=10em]
    \tikzstyle{level 3}=[sibling distance=10em]
    \node [] { $*$ }
        child [] { node [] { $*$ } 
            child [] { node [] { 2 }
                }
            child [] { node [] { $\pi$ } 
                } 
            }
        child [] { node [] { $\sqrt{}$ } 
            child [missing] { node [] { }
                }
            child [] { node [] { $/$ } 
                child [] { node [] { $L$ }
                    }
                child [] { node [] { $g$ } 
                    } 
                } 
            };
    \end{tikzpicture}
    \caption{An expression tree of the pendulum formula, $T=2\pi\sqrt{\frac{L}{g}}$.}
    \label{Fig:Expr}
\end{figure}

The objective in \cref{eq:SymbRegr} is to minimize the empirical loss, in other words, to maximize the accuracy of the expression.
Additionally, the objective function may include a regularization term modeling the complexity of the expression
to obtain a simple expression.
A common way to measure the complexity of symbolic regression is to calculate the number of nodes in an expression tree.
We review in \cref{sec:sr-formulation} how we can formulate \cref{eq:SymbRegr} by modeling expression
trees using binary variables. 





\subsection{Methods, Test Problems, and Challenges for Symbolic Regression} 
\label{sec:sr-literature}

Over the past few years, researchers have expressed renewed interest in learning symbolic expressions. Both exact formulations and
solution techniques based on mixed-integer nonlinear programming  and heuristic search techniques have been proposed.
Recently, machine learning methodologies and hybrid techniques have also been proposed. Below, we briefly review
each class of methods as well as test problem collections.

\paragraph{Heuristic Techniques for Symbolic Regression}
Genetic programming  is the most common approach for solving symbolic regression. It 
begins with an initial population of individuals or randomly selected expression trees. 
These trees are compared by a fitness measure and error metrics. 
Individuals with high scores have a higher probability of being selected for the next iteration of crossover, mutation, and reproduction.
Ideas to enhance genetic algorithms have been proposed in \cite{kammerer2020symbolic, kronberger2019cluster} to reduce the search space.
Nicolau and McDermott~\cite{nicolau2020genetic} use prior information of the values of the dependent variable.
The quality of solutions is not stable, however, because genetic algorithms are
a stochastic process, which means that it can generate different solutions for the same input data and the same settings. Kammerer et al.~\cite{kammerer2020symbolic} remark that ``it might produce highly dissimilar solutions even for the same input data.''
%

\paragraph{Exact Mixed-Integer Approaches to Symbolic Regression}
Exact approaches based on the MINLP formulation are deterministic in the sense that they return the same solution if the input data and parameter 
settings are the same. In principle MINLP approaches are exact in the sense that they will recover the global solution of \cref{eq:SymbRegr},
although their runtime may be prohibitively long in practice.
MINLP formulations were first proposed in \cite{cozad2014data}, extended
in \cite{cozad2018global, neumann2019new}, and independently studied in \cite{austel2017globally, horesh2016globally}.
MINLP formulations use binary variables to define the expression tree and continuous variables to represent intermediate values for 
each node and each data point; see \cref{sec:sr-formulation}. The resulting optimization problem is a nonlinear, nonconvex MINLP, because it involves nonlinear operators 
such as $*, /, \exp$, and $\log$.
The solution time typically increases exponentially in the maximum depth of the tree.
To limit the runtime of the MINLP solvers, several researchers \cite{austel2017globally, cozad2014data, cozad2018global, neumann2019new} limit the structure of the expression tree 
(usually by limiting its depth) and solve a smaller MINLP. The approaches in \cite{austel2017globally, cozad2014data, cozad2018global} are tested only on noiseless data.
Neumann et al. \cite{neumann2019new} propose an interesting methodology to avoid overfitting: 
(1) they add a constraint to limit the complexity of the expression tree (number of nodes);
(2) they then generate a portfolio of solutions by varying the complexity limitation on the training set; and 
(3) they choose the solution based on the
validation error.

\paragraph{Approaches Based on Machine Learning Methodologies} 
An approach based on training a neural network has been proposed in \cite{udrescu2020ai2,udrescu2020ai} and is referred to as 
AI Feynman. 
A key of AI Feynman is to discover functional properties of the overall function using a neural network in order to reduce the search space and decompose the overall symbolic regression problem into smaller subproblems.
The method recursively applies dimension reduction techniques (dimensional analysis, symmetry, and separability detection) until the remaining components are simple enough to be detected by polynomial fits or complete enumeration. 
These techniques require knowledge of the units of the independent and the dependent variables in order to perform the dimensional analysis, as well as smoothness of the underlying expression, because the method trains a neural network to evaluate values on missing points.
A related method is considered in \cite{cranmer2020discovering}, which involves using a graph neural network.


\paragraph{Hybrid Techniques for Symbolic Regression}
Austel et al. \cite{austel2020symbolic} propose an interesting idea that considers a generalized expression tree instead of an expression tree.
A generalized tree assigns a monomial ($hx_1^{a_1}x_2^{a_2} \cdots x_d^{a_d}$) to each leaf node, instead of a single variable or constant.
Consequently, generalized expression trees can represent a larger class of functions for the same depth.

The algorithm has two steps.
First, it lists all generalized trees up to depth $D$.
The number of generalized trees is reduced by removing redundant expressions.
For example, both the generalized tree corresponding to the summation of two monomials ($hx_1^{a_1}x_2^{a_2} \cdots x_d^{a_d} + gx_1^{b_1}x_2^{b_2} \cdots x_d^{b_d}$) and the generalized tree corresponding to the subtraction of two monomials ($h'x_1^{a'_1}x_2^{a'_2} \cdots x_d^{a'_d} - g'x_1^{b'_1}x_2^{b'_2} \cdots x_d^{b'_d}$) are equivalent because both can represent the same set of functions; therefore, one of the expressions can be removed from the list.
The list of all generalized trees up to depth one with operators $\{+,-,*,/,\sqrt{}\}$ is
\[
L_1, ~\sqrt{L_1}, ~(L_1+L_2 \equiv L_1 - L_2), ~(L_1 * L_2 \equiv L_1 / L_2),
\]
where $L_1$ and $L_2$ are monomials and $\equiv$ represents that two expressions are equivalent.
Second, the algorithm solves optimization problems to find a global solutions for each generalized tree.
For example, the optimization problem of $L_1+L_2$ is 
\begin{align}
    \min_{h, a_1,\dots,a_d, g, b_1,\dots,b_d} & \sum_{i=1}^{n_{\text{data}}} \left(z_i - \left(hx_{i,1}^{a_1}x_{i,2}^{a_2} \cdots x_{i,d}^{a_d} + gx_{i,1}^{b_1}x_{i,2}^{b_2} \cdots x_{i,d}^{b_d}\right)\right)^2 ,
\end{align}
where $h,g$ are bounded continuous variables and $a,b$ are bounded integer variables.
Even though the problem assumes that the tree structure of the generalized expression tree is given, it is a mixed-integer nonlinear (nonconvex) programming problem that is in NP-hard.
Constraints are added based on the knowledge of the units of the independent and the dependent variables to reduce the search space.


\paragraph{Benchmark Problems for Symbolic Regression}
Several test problem collections have been produced to test symbolic regression ideas. For example, ALAMO, the Automatic Learning of Algebraic Models, 
is a software package for symbolic regression; see  \url{http://minlp.com/alamo}. The test set of \cite{cozad2018global} (see \url{http://minlp.com/nlp-and-minlp-test-problems}) has 24 instances. 
In \cite{austel2017globally}, the authors consider Kepler's law $d = \sqrt[3]{c\tau^2 (M + m)}$ and the period of the pendulum 
$\tau = 2\pi \sqrt{\ell / g}$. A set of examples from \cite{white2013better} is available at \url{http://gpbenchmarks.org}.
A number of test problems are also in \cite{udrescu2020ai2, udrescu2020ai} and are available at \url{https://space.mit.edu/home/tegmark/aifeynman.html}.

\paragraph{The Challenges of Symbolic Regression}

Solving symbolic regression problems such as \cref{eq:SymbRegr} has been shown to be challenging; see, for example,\cite{austel2020symbolic, austel2017globally, cozad2018global,neumann2019new}. 
In particular, the problem complexity increases with the number of operators and the number of operator types.
For example, the number of expression trees represented by $d$ independent variables and $B$ binary operators with a maximum depth $D$ is more than $(B \cdot d)^{2^{D}}/B$.\footnote{Let $T_{\delta}$ denote the number of expression trees up to depth $\delta$.
$T_D \ge (B \cdot d)^{2^{D}}/B$ is derived from the system of $T_0 \ge n$ and $T_{\delta} \ge B \cdot T_{\delta-1}^2$. The equality holds if we allow an expression tree to use only binary operators and the independent variables.} 
Another challenge is the non-convexity of the problem, which remains even if the expression tree is fixed because the problem with a fixed expression tree is equivalent to optimize parameters of an arbitrary functional form. 

\section{An Improved MINLP Formulation of Symbolic Regression} \label{sec:sr-formulation}

We review and improve a MINLP formulation that searches an expression tree with the minimum training error given data points $(x_{i,1},\dots,x_{i,d},z_i) \in \R^{d+1}$ for $i=1,\dots,n_{\text{data}}$.
Our formulation improves the one proposed by \cite{cozad2018global}.
The new constraints remove equivalent expression trees and tighten the feasible set of the relaxation.

\subsection{Notation}
A \emph{relaxation} refers to the relaxation obtained by relaxing binary variables.
We let $[a] := \{1,2,\dots,a\}$ for $a \in \mathbb{Z}_{++}$, the set of positive integers.

\subsection{Inputs}
We are given a set of operators and a set of nodes that define the superset of all feasible expression trees.
By limiting the number of nodes and operands used to construct the expression tree, we transform the infinite-dimensional problem \cref{eq:SymbRegr} into a finite-dimensional problem.
We denote by $\mathcal{P} \subseteq \{+, -, *, /, \sqrt{}, \exp, \log, \dots\}$ a set of operators.
To streamline our presentation, we define the set of binary operators $\mathcal{B} := \mathcal{P} \cap \{+, -, *, /\}$, the set of unary operators $\mathcal{U} := \mathcal{P} \cap \{\sqrt{}, \exp, \log\}$, and the set of operands $\mathcal{L} = \{x_1,\dots,x_d,\text{cst}\}$, where cst is a constant.
We denote the set of all operators and operands by $\mathcal{O} := \mathcal{B} \cup \mathcal{U} \cup \mathcal{L}$.

We identify each node of the tree by an integer, and we let $\mathcal{N}$ denote the set of all nodes in the tree.
We 
denote the children of node $n \in \mathcal{N}$ by $2n$ and $2n+1$, respectively.
We assume $1 \in \mathcal{N}$, which is the root of the tree.
We denote by $\mathcal{T}$ the set of terminal nodes (that have no child).
We assume that $\mathcal{N}$ corresponds to a full binary tree.\footnote{A full (proper) binary tree is a tree in which every node has zero or two children.}
For example, $\mathcal{N} = \{1,2,3,6,7\}$ is a full binary tree, while $\mathcal{N} = \{1,2,3,4\}$ is not because node 2 has only one child, namely, node 4. 
Provided that the set of operators $\mathcal{O}$ is finite and we limit the number of nodes, this MINLP formulation transforms the infinite-dimensional functional approximation \cref{eq:SymbRegr} into a finite-dimensional problem.

\subsection{Decision Variables and Individual Variable Restrictions}
There are three types of decision variables.
The binary variable $y^o_n$ is one if operator $o$ is assigned to node $n$, and zero otherwise.
Variable $c_n$ is the constant value at node $n$ if node $n$ exists, and  zero otherwise.
Therefore, $y^o_n$ and $c_n$ determine the expression tree.
Variable $v_{i,n}$ represents the intermediate computation value at node $n$ for data point $i$.
In other words, $v_{i,n}$ is the value of the symbolic expression represented by the subtree rooted by node $n$ at data point $i$.
Therefore, $v_{i,1}$ is the value predicted by the expression tree of data point $i$.
All continuous variables are bounded.
To streamline our presentation, we use $n \notin \mathcal{T}$ instead of $n \in \mathcal{N}\setminus \mathcal{T}$ in the following discussions.
We define $\mathcal{Y} := \{(n,o), ~\forall o \in \mathcal{O}, ~\forall n \notin \mathcal{T}\} \cup \{(n,o), ~\forall o \in \mathcal{L}, ~\forall n \in \mathcal{T}\}$, the set of all pairs of node $n$ and operator $o$ such that $o$ can be assigned to $n$.
To summarize, our model has the following set of variables and ranges:
\begin{subequations}
\begin{align}
    & y_{n}^{o} \in \{0,1\}, && \forall (n,o) \in \mathcal{Y}, 
    \label{constr:main:y} \\
    & c_{\text{lo}} \le c_{n} \le c_{\text{up}}, && \forall n \in \mathcal{N}, \label{constr:main:c} \\
    & v_{\text{lo}} \le v_{i,n} \le v_{\text{up}}, && \forall i \in [n_{\text{data}}], ~\forall n \in \mathcal{N}. \label{constr:main:v}
\end{align}
\end{subequations}
We assume without loss of generality that $v_{\text{lo}} \le c_{\text{lo}} \le 0 \le c_{\text{up}} \le v_{\text{up}}$.

\subsection{Objective Function}
We minimize the mean of the squared errors
\begin{align}
\min \quad \frac{1}{n_{\text{data}}}\sum_{i=1}^{n_{\text{data}}} (z_i - v_{i,1})^2. \label{obj:main}
\end{align}
Additionally, we might add a regularization term such as $\lambda \sum_{(n,o) \in \mathcal{Y}} y^o_n$, where $\lambda \in \R_+$ is a regularization parameter to promote a sparser expression tree.

\subsection{Constraints}

In \cref{sec:treedef,sec:valuedef}, we introduce constraints that are necessary to solve this problem.
In \cref{sec:redundancy,sec:implication,sec:symmetry}, we introduce constraints that remove equivalent expression trees and/or reduce the space of the relaxation, an acton that potentially leads to an improvement in computation.
We compare the constraints with the similar constraints in \cite{cozad2018global} in each section.
For the sake of completeness, the formulation proposed in \cite{cozad2018global} is summarized in \cref{sec:cozad-formulation} in terms of our notations.

\subsubsection{Tree-Defining Constraints}
\label{sec:treedef}

Tree-defining constraints enforce that the assignment of operators and operands results in a valid expression tree.
The constraints consist of \cref{constr:grammar1,constr:grammar2,cozad:grammar1,cozad:grammar2}.
In addition to \cref{cozad:grammar1,cozad:grammar2} from \cite{cozad2018global},
we use the following tree-defining constraints: 
\begin{subequations}
\begin{align}
    &&& \sum_{o \in \mathcal{B} \cup \mathcal{U}} y^o_n = \sum_{o \in \mathcal{O}} y^o_{2n+1}, && n \notin \mathcal{T}, \label{constr:grammar1}\\
    &&& \sum_{o \in \mathcal{B}} y^o_n = \sum_{o \in \mathcal{O}} y^o_{2n}, && n \notin \mathcal{T}. \label{constr:grammar2} 
\end{align}
\end{subequations}
Constraint~\cref{constr:grammar1} enforces that a binary/unary operator is assigned to node $n$ if and only if its right child (node $2n+1$) exists.
Constraint~\cref{constr:grammar2} enforces that a binary operator is assigned to node $n$ if and only if its left child (node $2n$) exists.
Constraint~\cref{cozad:grammar1} forces the assignment of at most one operator to a node.
Constraint~\cref{cozad:grammar2} forces the expression tree to include at least one independent variable.
All four constraints are necessary to obtain an expression tree with valid operator/operand assignments.

In contrast, the tree-defining constraints from \cite{cozad2018global} are \cref{cozad:grammar1}-\cref{cozad:grammar6}.
\cref{fig:proof-grammar} shows two assignments of operators and operands to an expression tree.
Both represent the symbolic expression $x_1$, and
both are feasible in \cref{cozad:grammar3}-\cref{cozad:grammar6}, but only \cref{fig:proof-grammar1} is feasible in \cref{constr:grammar1}-\cref{constr:grammar2}.
\begin{figure}[htb]
    \centering
    \subfloat[]{
        \label{fig:proof-grammar1}
        \begin{tikzpicture}
        [scale=0.5,
        every node/.style = {shape=circle, align=center, draw, line width=0.2mm, minimum size = 7mm
        },
        edge from parent/.style = {draw, line width=0.2mm},
        level distance=5em,
        ]
        \tikzstyle{level 1}=[sibling distance=15em]
        \tikzstyle{level 2}=[sibling distance=10em]
        \tikzstyle{level 3}=[sibling distance=10em]
        \node [] { $x_1$ }
            child [] { node [] { } 
                child [] { node [] {  }
                    }
                child [] { node [] {  } 
                    } 
                }
            child [] { node [] {  } 
                child [] { node [] { }
                    }
                child [] { node [] { } 
                    } 
                };
        \end{tikzpicture}
        }
    \subfloat[]{
        \label{fig:proof-grammar2}
        \begin{tikzpicture}
        [scale=0.5,
        every node/.style = {shape=circle, align=center, draw, line width=0.2mm, minimum size = 7mm
        },
        edge from parent/.style = {draw, line width=0.2mm},
        level distance=5em,
        ]
        \tikzstyle{level 1}=[sibling distance=15em]
        \tikzstyle{level 2}=[sibling distance=10em]
        \tikzstyle{level 3}=[sibling distance=10em]
        \node [] { $x_1$ }
            child [] { node [] { $x_2$ } 
                child [] { node [] {  }
                    }
                child [] { node [] {  } 
                    } 
                }
            child [] { node [] {  } 
                child [] { node [] { }
                    }
                child [] { node [] { $x_3$ } 
                    } 
                };
        \end{tikzpicture}
        }
    \caption{Two equivalent expression trees corresponding to $x_1$. 
    Both are feasible in the \cite{cozad2018global}'s formulation, while only (a) is feasible in our improved formulation.
    }
    \label{fig:proof-grammar}
\end{figure}
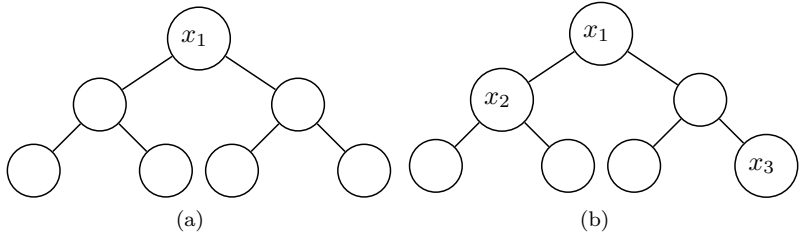
We formalize this observation by showing in \cref{lem:grammer} that our tree-defining constraints system is tighter.

\begin{lemma} \label{lem:grammer}
It holds that 
$\left\{y \in \{0,1\}^{|\mathcal{Y}|} ~|~  \cref{constr:grammar1}, \cref{constr:grammar2}, \cref{cozad:grammar1}, \cref{cozad:grammar2} \right\} \allowbreak \subsetneq \allowbreak \{y \in \{0,1\}^{|\mathcal{Y}|} ~|~  \allowbreak \cref{cozad:grammar1}\mhyphen\cref{cozad:grammar6}\}$.
\end{lemma}

\begin{proof}
Let $S := \left\{y ~|~ \cref{constr:main:y}, \cref{constr:grammar1}, \cref{constr:grammar2}, \cref{cozad:grammar1}, \cref{cozad:grammar2} \right\}$ and $T := \{y ~|~ \cref{constr:main:y}, \allowbreak \cref{cozad:grammar1}\mhyphen\cref{cozad:grammar6}\}$.
We first show that $S \subseteq T$.
Pick any point $y \in S$.
We need to show that $y$ satisfies \cref{cozad:grammar3}-\cref{cozad:grammar6}.
Point $y$ satisfies \cref{cozad:grammar3} and \cref{cozad:grammar4} because those are relaxations of \cref{constr:grammar1} and \cref{constr:grammar2}, respectively, by the fact that $\mathcal{B}$ (set of binary operators), $\mathcal{U}$ (set of unary operators), and $\mathcal{L}$ (set of operands) are a partition of $\mathcal{O}$ (set of operators and operands).
We can also show that $y$ satisfies \cref{cozad:grammar5} by
\[
    \sum_{o \in \mathcal{U} \cup \mathcal{L}} y^o_n \le 1 - \sum_{o \in \mathcal{B}} y^o_n = 1 - \sum_{o \in \mathcal{O}} y^o_{2n} ,
\]
where the inequality and the equality hold by \cref{cozad:grammar1} and \cref{constr:grammar2}, respectively.
Similarly, we can show that $y$ satisfies \cref{cozad:grammar6} by 
\[
    \sum_{o \in \mathcal{L}} y^o_n \le 1 - \sum_{o \in \mathcal{B} \cup \mathcal{U}} y^o_n = 1 - \sum_{o \in \mathcal{O}} y^o_{2n+1} ,
\]
where the inequality and the equality hold by \cref{cozad:grammar1} and \cref{constr:grammar1}, respectively.
Therefore, $y$ satisfies all the constraints in $T$, and consequently $S \subseteq T$ holds.

We next show that there exists $y \in T\setminus S$.
Let $\mathcal{N} = [7]$, $y_1^{x_1}=y_2^{x_2}=y_7^{x_3}=1$, and otherwise $y_n^o=0$.
\Cref{fig:proof-grammar2} shows the expression tree corresponding to $y$.
The $y$ satisfies \cref{cozad:grammar1} and \cref{cozad:grammar2}.
Also, $y$ satisfies \cref{cozad:grammar1}--\cref{cozad:grammar2} because every left-hand-side value is zero.
However, $y$ does not satisfy \cref{constr:grammar1} and \cref{constr:grammar2} because $x_2$ and $x_3$ cannot be assigned unless an operator is assigned to their parents.
Therefore, the proof is complete.
\end{proof}

\subsubsection{Value-Defining Constraints}
\label{sec:valuedef}

Value-defining constraints enforce that the value of $v_{i,n}$ is computed based on the solution expression tree and data points.
Specifically, if an operand is assigned to node $n$, then $v_{i,n}$ is equal to the value of the operand.
If a unary operator $\otimes(x)$ is assigned to node $n$, then $v_{i,n}$ is equal to $\otimes(v_{i,2n+1})$.
If a binary operator $\otimes$ is assigned to node $n$, then $v_{i,n}$ is equal to $v_{i,2n}\otimes v_{i,2n+1}$.
Cozaad and Sahinidis \cite{cozad2018global} introduce a value-defining constraint for each data point, each node, and each operator or operand, given in \cref{eqn:nonelb}--\cref{eqn:logdomain} in \cref{sec:form-valueconstr}.
All these constraints are necessary to ensure the correct prediction values for each data point given an expression tree.

We propose a set of improved value-defining constraints, \cref{eqn:varub}--\cref{eqn:varlb} together with \cref{eqn:cstub}--\cref{eqn:logdomain} from \cite{cozad2018global}:
\begin{subequations}
\begin{align}
    &&& v_{i,n} \le \sum_{j=1}^d x_{i,j}y^{x_j}_n + v_{\text{up}}\sum_{o \in \mathcal{B} \cup \mathcal{U} \cup \{\text{cst}\}} y^{o}_n, && \forall i\in [n_{\text{data}}], ~\forall n \in \mathcal{N}, \label{eqn:varub}\\
    &&& v_{i,n} \ge \sum_{j=1}^d x_{i,j}y^{x_j}_n + v_{\text{lo}}\sum_{o \in \mathcal{B} \cup \mathcal{U} \cup \{\text{cst}\}} y^{o}_n, && \forall i\in [n_{\text{data}}], ~\forall n \in \mathcal{N}. \label{eqn:varlb}
\end{align}
\end{subequations}
Our value-defining constraints replace \cref{eqn:noneub}-\cref{eqn:indeplb} with \cref{eqn:varub}-\cref{eqn:varlb}.
Both \cref{eqn:varub}-\cref{eqn:varlb} and \cref{eqn:noneub}-\cref{eqn:indeplb} represent the following disjunction:
\begin{multline}
    \bigvee_{o \in \{x_1,\dots,x_d\}} 
    \left[\begin{array}{c}
        y_{n}^o=1 \\
        v_{i,n} = x_{i,j}, \\
        \forall i \in [n_{\text{data}}]
    \end{array}\right]
    \bigvee
    \left[\begin{array}{c}
        \sum_{o \in \mathcal{O}} y_{n}^o = 0 \\
        v_{i,n} = 0, \\ 
        \forall i \in [n_{\text{data}}]
    \end{array}\right]
    \bigvee
    \left[\begin{array}{c}
        \sum_{o \in \mathcal{B}\cup\mathcal{U}\cup\{\text{cst}\}} y_{n}^o = 1 \\
        v_{\text{lo}} \le v_{i,n} \le v_{\text{up}}, \\ 
        \forall i \in [n_{\text{data}}]
    \end{array}\right], \\
    \forall n \in \mathcal{N}.
\end{multline}
Note that if $\sum_{o \in \mathcal{B}\cup\mathcal{U}\cup\{\text{cst}\}} y_{n}^o = 1$ (i.e., node $n$ exists), then the value of $v_{i,n}$ is determined by \cref{eqn:cstub}--\cref{eqn:logdomain}.
Our formulation reduces the number of constraints.
The number of constraints \cref{eqn:varub}--\cref{eqn:varlb} is $n_{\text{data}}|\mathcal{N}|$, while the number of constraints \cref{eqn:noneub}--\cref{eqn:indeplb} is $n_{\text{data}}|\mathcal{N}|(d+1)$.
We show in \cref{lem:valuedef} that our formulation does not change the feasible set and, in fact, reduces the space of the relaxation.

We first show in  \cref{lem:multichoice} that we can merge $k$ big-$M$ constraints associated with different constant bounds on an identical function.
This merge reduces the space of the relaxation while it does not change the feasible space.

\begin{lemma} \label{lem:multichoice}
Let $m$ be a positive integer, $k \in [m]$, $w \in \mathbb{R}^k$, $M >
\max_{i \in [k]} w_i$.
Let $\mathcal{F}_B = \{y \in \{0,1\}^m ~|~ \sum_{i=1}^m y_i = 1\}$ and $\mathcal{F}_C = \{y \in \mathbb{R}_+^m ~|~ \sum_{i=1}^m y_i = 1\}$.
Consider sets $S$ and $T$, where
\begin{align*}
    S & := \{(x,y) \in \mathbb{R} \times \mathbb{R}^n ~|~ f(x) \le \sum_{i \in [k]} w_iy_i + M(1 - \sum_{i \in [k]} y_i)\}, \\
    T & := \{(x,y) \in \mathbb{R} \times \mathbb{R}^n ~|~ f(x) \le w_iy_i + M(1-y_i), ~\forall i \in [k]\}. 
\end{align*}
Then, the following relations hold:
\begin{subequations}
\begin{align}
& \{(x,y) \in S ~|~ y \in \mathcal{F}_B\} = \{(x,y) \in T ~|~ y \in \mathcal{F}_B\}, \label{lem:multichoice1}\\
& \{(x,y) \in S ~|~ y \in \mathcal{F}_C\} \subsetneq \{(x,y) \in T ~|~ y \in \mathcal{F}_C\}. \label{lem:multichoice2}
\end{align}
\end{subequations}
\end{lemma}

\begin{proof}
The proof of this result is given in \cref{lem:multichoice:proof}.
\end{proof}

Next, we show that the new constraints do not change the feasible set of the MINLP 
but improve its continuous relaxation.

\begin{lemma} \label{lem:valuedef}
Let $\mathcal{F}_B = \{(y,v) \in \{0,1\}^{|\mathcal{Y}|} \times [v_{\text{lo}},v_{\text{up}}]^{n_{\text{data}}|\mathcal{N}|} ~|~ \cref{cozad:grammar1}\}$ and $\mathcal{F}_C = \{(y,v) \in [0,1]^{|\mathcal{Y}|} \times [v_{\text{lo}},v_{\text{up}}]^{n_{\text{data}}|\mathcal{N}|} ~|~ \cref{cozad:grammar1}\}$.
It holds that
\begin{subequations}
\begin{align}
\{(y,v) \in \mathcal{F}_B ~|~ \cref{eqn:varub}\mhyphen\cref{eqn:varlb} \} &= \{(y,v) \in \mathcal{F}_B ~|~ \cref{eqn:noneub}\mhyphen\cref{eqn:indeplb}\}, \label{eqn:valuedef1} \\    
\{(y,v) \in \mathcal{F}_C ~|~ \cref{eqn:varub}\mhyphen\cref{eqn:varlb} \} & \subsetneq \{(y,v) \in \mathcal{F}_C ~|~ \cref{eqn:noneub}\mhyphen\cref{eqn:indeplb}\}. \label{eqn:valuedef2}
\end{align}
\end{subequations}
\end{lemma}

\begin{proof}
Constraints \cref{cozad:grammar1}, \cref{eqn:varub}--\cref{eqn:varlb}, and \cref{eqn:noneub}--\cref{eqn:indeplb} are all separable in $n \in \mathcal{N}$.
Thus, it is sufficient to show that \cref{eqn:valuedef1} and \cref{eqn:valuedef2} hold for a specific $n$.
We introduce $y^{\text{none}}_n := 1 - \sum_{(n',o) \in \mathcal{Y} : n' = n} y_n^o$.
By definition, $y^{\text{none}}_n + \sum_{(n',o) \in \mathcal{Y} : n' = n} y_n^o = 1$.
Then, we can merge all ``$v_{i,n} \le \cdots$'' constraints in \cref{eqn:noneub}--\cref{eqn:indeplb} into \cref{eqn:varub} for all $n \in \mathcal{N}$, an action that corresponds to merging constraints in $T$ to the constraint in $S$ in \cref{lem:multichoice}.
Similarly, we merge all ``$v_{i,n} \ge \cdots$'' constraints in \cref{eqn:noneub}--\cref{eqn:indeplb} into \cref{eqn:varlb} for all $n \in \mathcal{N}$.
By \cref{lem:multichoice}, this merge strictly reduces the space of the relaxation while it does not change the feasible set.
Therefore, the proof is complete.
\end{proof}

\subsubsection{Redundancy-Eliminating Constraints}
\label{sec:redundancy}

Redundancy-eliminating constraints exclude redundant operations and remove three kinds of redundancy described in \cref{tab:redundancy}.
\begin{table}[htb]
    \centering
    \begin{tabular}{c|c|c} \hline
    Redundancy type             &  Example & Constraints \\ \hline
    Association property & $x_1 - (x_2 - 3) = x_1 + (3-x_2)$ & \cref{constr:redun2}, \cref{constr:redun3} \\
    Operations on constants     &  $2 = \sqrt{4} = 1.5 + 0.5$ & \cref{constr:redun1}, \cref{cozad:redun4} \\
    Nested operations & $x = e^{\log(x)} = \log(e^x)$ & \cref{cozad:redun5}, \cref{cozad:redun6} \\ \hline
    \end{tabular}
    \caption{Redundancy removing constraints.}
    \label{tab:redundancy}
\end{table}

In addition to the redundancy-eliminating constraints \cref{cozad:redun4}-\cref{cozad:redun6}
from \cite{cozad2018global}, we introduce the following constraints:
\begin{subequations}
\begin{align}
    &&& y^{+}_{n} + y^{-}_{2n+1} \le 1, && n \notin \mathcal{N}_{\text{perfect}}, \label{constr:redun2} \\    
    &&& y^{*}_{n} + y^{/}_{2n+1} \le 1, && n \notin \mathcal{N}_{\text{perfect}}, \label{constr:redun3} \\
    &&& y^{\text{cst}}_{2n+1} \le y^+_n + y^*_n, && n \notin \mathcal{T}, \label{constr:redun1}
\end{align}
\end{subequations}
where $\mathcal{N}_{\text{perfect}}$ is a set of nodes whose rooted subtree of $\mathcal{N}$ is a perfect binary tree.\footnote{A perfect binary tree is a binary tree in which all nonterminal nodes have two children and all terminal nodes have the same depth.}
Specifically, the constraints exclude all equivalent expressions except the first expression of
the examples in \cref{tab:redundancy}.

Redundancy induced by the association property is not considered in \cite{cozad2018global}, while our redundancy-removing constraints \cref{constr:redun2} and \cref{constr:redun3} exclude such cases.  
For example, all the expression trees in \cref{fig:redun-new12,fig:redun-new34} are feasible in the formulation in \cite{cozad2018global}; however, only (a) is feasible in our formulation.

Constraint~\cref{constr:redun1} allows the right child to be a constant only if $+$ or $*$ is assigned to node $n$.
The constraints  exclude $- C$ and $/ C$ because equivalent expressions $+ (-C)$ and $* (1/C)$ are feasible.
In addition, they exclude $\sqrt{C}$, $\exp{C}$, and $\log{C}$ because we can represent them as a single constant node.  
This type of redundancy is considered in \cite{cozad2018global}, and the corresponding constraints are \cref{cozad:redun1}--\cref{cozad:redun3}.
We show in \cref{lem:redundancy} that the substitution  \cref{constr:redun1} for \cref{cozad:redun1}--\cref{cozad:redun3} results in a tighter relaxation.

\begin{figure}[htb]
    \centering
    \subfloat[$A - (B - C)$]{
        \begin{tikzpicture}
        [scale=0.5,
        every node/.style = {shape=circle, align=center, draw, line width=0.2mm, minimum size = 7mm
        },
        edge from parent/.style = {draw, line width=0.2mm},
        level distance=5em,
        ]
        \tikzstyle{level 1}=[sibling distance=15em]
        \tikzstyle{level 2}=[sibling distance=10em]
        \tikzstyle{level 3}=[sibling distance=10em]
        \node [] { $-$ }
            child [] { node [] { $A$ } 
                }
            child [] { node [] { $-$ } 
                child [] { node [] { $B$ }
                    }
                child [] { node [] { $C$ } 
                    } 
                };
        \end{tikzpicture}
    }
    \subfloat[$A + (C - B)$]{
        \begin{tikzpicture}
        [scale=0.5,
        every node/.style = {shape=circle, align=center, draw, line width=0.2mm, minimum size = 7mm
        },
        edge from parent/.style = {draw, line width=0.2mm},
        level distance=5em,
        ]
        \tikzstyle{level 1}=[sibling distance=15em]
        \tikzstyle{level 2}=[sibling distance=10em]
        \tikzstyle{level 3}=[sibling distance=10em]
        \node [] { $+$ }
            child [] { node [] { $A$ } 
                }
            child [] { node [] { $-$ } 
                child [] { node [] { $C$ }
                    }
                child [] { node [] { $B$ } 
                    } 
                };
        \end{tikzpicture}
    }
    \caption{Two equivalent expression trees with addition and subtraction. Both are feasible in \cite{cozad2018global}, but only (a) is feasible in \cref{constr:redun2}.}
    \label{fig:redun-new12}
\end{figure}
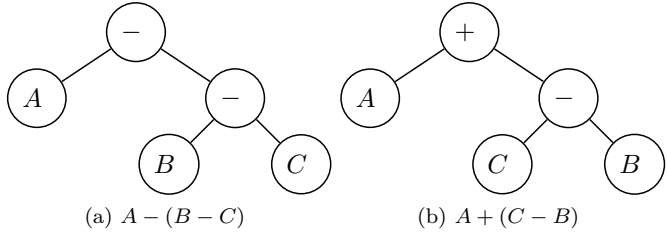

\begin{figure}[htb]
    \centering
    \subfloat[$A / (B / C)$]{
        \begin{tikzpicture}
        [scale=0.5,
        every node/.style = {shape=circle, align=center, draw, line width=0.2mm, minimum size = 7mm
        },
        edge from parent/.style = {draw, line width=0.2mm},
        level distance=5em,
        ]
        \tikzstyle{level 1}=[sibling distance=15em]
        \tikzstyle{level 2}=[sibling distance=10em]
        \tikzstyle{level 3}=[sibling distance=10em]
        \node [] { $/$ }
            child [] { node [] { $A$ } 
                child [missing] { node [] {  }
                    }
                child [missing] { node [] {  } 
                    } 
                }
            child [] { node [] { $/$ } 
                child [] { node [] { $B$ }
                    }
                child [] { node [] { $C$ } 
                    } 
                };
        \end{tikzpicture}
    }
    \subfloat[$A * (C / B)$]{
        \begin{tikzpicture}
        [scale=0.5,
        every node/.style = {shape=circle, align=center, draw, line width=0.2mm, minimum size = 7mm
        },
        edge from parent/.style = {draw, line width=0.2mm},
        level distance=5em,
        ]
        \tikzstyle{level 1}=[sibling distance=15em]
        \tikzstyle{level 2}=[sibling distance=10em]
        \tikzstyle{level 3}=[sibling distance=10em]
        \node [] { $*$ }
            child [] { node [] { $A$ } 
                }
            child [] { node [] { $/$ } 
                child [] { node [] { $C$ }
                    }
                child [] { node [] { $B$ } 
                    } 
                };
        \end{tikzpicture}  
    }
    \caption{Two equivalent expression trees with multiplication and division. Both are feasible in \cite{cozad2018global}, but only (a) is feasible in \cref{constr:redun3}.}
    \label{fig:redun-new34}
\end{figure}
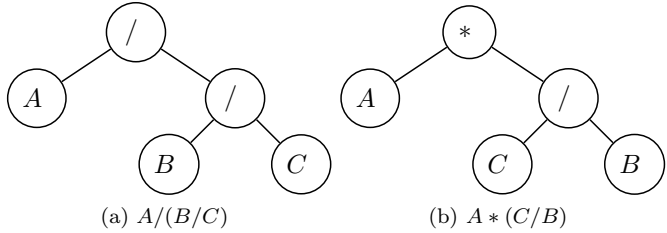

\begin{lemma} \label{lem:redundancy}
Let $\mathcal{F}_B$ and $\mathcal{F}_C$ denote the set of binary and continuous $y$, respectively,
that satisfy the tree-defining constraints, that is,
\begin{subequations}
\begin{align*}
    \mathcal{F}_B & = \{y \in \{0,1\}^{|\mathcal{Y}|} ~|~ \cref{constr:grammar1}, \cref{constr:grammar2}, \cref{cozad:grammar1}, \cref{cozad:grammar2}\}, \\
    \mathcal{F}_C & = \{y \in [0,1]^{|\mathcal{Y}|} ~|~ \cref{constr:grammar1}, \cref{constr:grammar2}, \cref{cozad:grammar1}, \cref{cozad:grammar2}\}.
\end{align*}
\end{subequations}
It follows that 
\begin{subequations}
\begin{align}
        \{y \in \mathcal{F}_B ~|~ \cref{constr:redun1}\} & = \{y \in \mathcal{F}_B ~|~ \mbox{\cref{cozad:redun2}-\cref{cozad:redun4}}\}, \\
        \{y \in \mathcal{F}_C ~|~ \mbox{\cref{constr:redun1}}\} & \subsetneq \{y \in \mathcal{F}_C ~|~ \mbox{\cref{cozad:redun2}-\cref{cozad:redun4}}\}.
\end{align}
\end{subequations}
\end{lemma}

\begin{proof}
Recall \cref{cozad:redun2}-\cref{cozad:redun4}:
\begin{align*}
    y^{\text{cst}}_{2n+1} & \le 1 - \sum_{o \in \mathcal{U}} y^o_n, && n \notin \mathcal{T}, \\
    y^{\text{cst}}_{2n+1}  & \le 1 - y^-_n, && n \notin \mathcal{T}, \\
    y^{\text{cst}}_{2n+1} & \le 1 -  y^/_n, && n \notin \mathcal{T}.
\end{align*}

By relaxing \cref{constr:grammar1}, we get the following inequality:
\begin{align*}
    & y^{\text{cst}}_{2n+1} \le 1 - y_n^{\text{none}} , && n \notin \mathcal{T},
\end{align*}
where $y_n^{\text{none}} := 1 - \sum_{o \in \mathcal{B} \cup \mathcal{U}} y^o_n$.
Let us consider that those constraints are defining the upper bound of $y_{2n+1}^{\text{cst}}$ depending on the choice of $y_n$. 
For example, $y^{\text{cst}}_{2n+1} \le 1 - y^-_n$ enforces the upper bound by 0 if $y^-_n = 1$; otherwise it relaxes this constraint. 
By \cref{lem:multichoice}, we can merge the constraints for each $n$.
The merged constraints are
\begin{align*}
    y^{\text{cst}}_{2n+1} & \le 1 - \sum_{o \in \mathcal{U}} y^o_n - y^-_n - y^/_n - y_n^{\text{none}} = y^+_n + y^*_n, && n \notin \mathcal{T}.
\end{align*}
By \cref{lem:multichoice}, this merge does not change the feasible set; it reduces the feasible set of the relaxation.
\end{proof}

\subsubsection{Implication Cuts}
\label{sec:implication}

Implication cuts are motivated by the fact that some operators are domain-restricted, for example, $/$, $\sqrt{}$, and $\log$.
From a set of given data points, we can identify  the characteristics of the independent variables:
\begin{subequations}
\begin{align}
    \mathcal{X}_{\text{posi}} & = \left\{i \in [d] ~\middle|~ \exists j \in [n_{\text{data}}]: x_{i,j} > 0\right\}, \\
    \mathcal{X}_{\text{nega}} & = \left\{i \in [d] ~\middle|~ \exists j \in [n_{\text{data}}]: x_{i,j} < 0\right\}, \\
    \mathcal{X}_{\text{zero}} & = \left\{i \in [d] ~\middle|~ \exists j \in [n_{\text{data}}]: x_{i,j} = 0 \right\}.
\end{align}
\end{subequations}
All invalid expression trees with up to depth one are described in \cref{fig:implication}.
\begin{figure}[htb]
    \centering
    \subfloat[$(\cdot) / x_{i}$, $i \in \mathcal{X}_{\text{zero}}$]{
        \begin{tikzpicture}
        [scale=0.6,
        every node/.style = {shape=circle, align=center, draw, line width=0.2mm, minimum size = 10mm
        },
        edge from parent/.style = {draw, line width=0.2mm},
        level distance=5em,
        ]
        \tikzstyle{level 1}=[sibling distance=12em]
        \tikzstyle{level 2}=[sibling distance=10em]
        \tikzstyle{level 3}=[sibling distance=10em]
        \node [] { $/$ }
            child [] { node [] { } 
                }
            child [] { node [] { $x_i$ } 
                };
        \end{tikzpicture}
    }
    \quad
    \subfloat[$\sqrt{x_j}$, $j \in \mathcal{X}_{\text{nega}}$]{
        \begin{tikzpicture}
        [scale=0.6,
        every node/.style = {shape=circle, align=center, draw, line width=0.2mm, minimum size = 10mm
        },
        edge from parent/.style = {draw, line width=0.2mm},
        level distance=5em,
        ]
        \tikzstyle{level 1}=[sibling distance=12em]
        \tikzstyle{level 2}=[sibling distance=10em]
        \tikzstyle{level 3}=[sibling distance=10em]
        \node [] { $\sqrt{}$ }
            child [] { node [] {  } 
                }
            child [] { node [] { $x_j$ } 
                };
        \end{tikzpicture}
    }
    \quad
    \subfloat[$\log(x_k)$, $k \in \mathcal{X}_{\text{nega}} \cup \mathcal{X}_{\text{zero}}$]{
        \begin{tikzpicture}
        [scale=0.6,
        every node/.style = {shape=circle, align=center, draw, line width=0.2mm, minimum size = 10mm
        },
        edge from parent/.style = {draw, line width=0.2mm},
        level distance=5em,
        ]
        \tikzstyle{level 1}=[sibling distance=12em]
        \tikzstyle{level 2}=[sibling distance=10em]
        \tikzstyle{level 3}=[sibling distance=10em]
        \node [] { $\log$ }
            child [] { node [] { } 
                }
            child [] { node [] { $x_k$ } 
                };
        \end{tikzpicture}
    }
    \caption{Invalid expression trees because of domain restriction.}
    \label{fig:implication}
\end{figure}
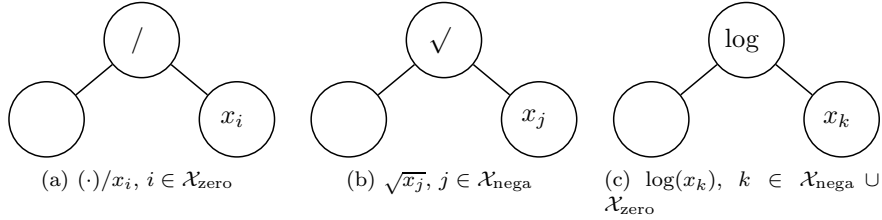
We can write the constraints that restrict the expressions in \cref{fig:implication} as follows:
\begin{subequations}
\begin{align}
    && y_{n}^{/} + y_{2n+1}^{x_j} & \le 1, && \forall j \in \mathcal{X}_{\text{zero}}, ~\forall n \notin \mathcal{T}, \label{constr:impl1} \\
    && y_{n}^{\sqrt{}} + y_{2n+1}^{x_j} & \le 1, && \forall j \in \mathcal{X}_{\text{nega}}, ~\forall n \notin \mathcal{T}, \label{constr:impl2} \\
    && y_{n}^{\log} + y_{2n+1}^{x_j} & \le 1, && \forall j \in \mathcal{X}_{\text{nega}} \cup \mathcal{X}_{\text{zero}}, ~\forall n \notin \mathcal{T}. \label{constr:impl4}
\end{align}
\end{subequations}
Constraints \cref{constr:impl1}--\cref{constr:impl4} exclude an expression tree that includes the expressions in \cref{fig:implication} as a subtree. 
We can generate more invalid trees from depth-two or higher-depth trees.
One example is $\sqrt{x_{j}x_{k}}$ for $j \in \mathcal{X}_{\text{posi}}$ and for $k \in \mathcal{X}_{\text{nega}}$.
The corresponding constraint is 
\begin{multline}
    y_{n}^{\sqrt{}} + y_{2n+1}^{*} + y_{4n+2}^{x_j} + y_{4n+3}^{x_k} \le 3, \\
    \forall j \in \mathcal{X}_{\text{posi}}, ~ \forall k \in \mathcal{X}_{\text{nega}}, 
    ~\forall n \in \mathcal{N} : \{4n+2, 4n+3\} \subseteq \mathcal{N}.
\end{multline}

Note that the expressions in \cref{fig:implication} are already infeasible in both our formulation and the formulation in \cite{cozad2018global} because of \cref{eqn:/domain_rch,eqn:sqrtdomain,eqn:logdomain}.
However, adding implication cuts reduces the feasible space of the relaxation.
Example~\ref{exam:implication} shows that the solution $y_n^{\sqrt{}}=y_{2n+1}^{x_j}=0.9$ for $j \in \mathcal{X}_{\text{nega}}$ is feasible in the space of the relaxation of the formulation without implication cuts, whereas it violates \cref{constr:impl2}.

\begin{example} \label{exam:implication}
We consider a symbolic regression problem with $\mathcal{N} = \{1,3\}$, $\mathcal{P} = \{\sqrt{}\}$ and a single data point $(x_{1,1},z_1) = (-1, 5)$.
Suppose that $v_{\text{lo}} = -10$, $v_{\text{up}} = 10$, and $\epsilon=0.01$. 
To streamline the presentation, we assume that $y_1^{\text{cst}} = y_3^{\text{cst}} = 0$.
The formulation without implication cuts is follows:
\begin{align*}
    \min \quad & (5-v_{1,1})^2, \\
    \mbox{s.t.} \quad & \mbox{Tree-Defining Constraints:} \\
    & y_1^{\sqrt{}} + y_1^{x_1} \le 1, \quad y_3^{x_1} \le 1, \quad y_1^{x_1} + y_3^{x_1} \ge 1, \quad y_1^{\sqrt{}} = y_3^{x_1}, \\
    & \mbox{Value-Defining Constraints:} \\
    & v_{1,1} \le (-1)y_1^{x_1} + 10y_1^{\sqrt{}}, \quad v_{1,1} \ge (-1)y_1^{x_1} - 10y_1^{\sqrt{}}, \\
    & v_{1,3} \le (-1)y_3^{x_1}, \quad v_{1,3} \ge (-1)y_3^{x_1}, \\
    & v_{1,1}^2 - v_{1,3} \le 90 (1-y_1^{\sqrt{}}), \quad v_{1,1}^2 - v_{1,3} \ge -10 (1-y_1^{\sqrt{}}), \\
    & 0.01 - v_{1,3} \le 10.01(1-y_1^{\sqrt{}}), \\
    & \mbox{Variable Restrictions:} \\
    & -10 \le v_{1,1}, v_{1,3} \le 10, \quad y_{1}^{\sqrt{}}, y_{1}^{x_1}, y_{3}^{x_1} \in \{0,1\}.
\end{align*}
Note that there is no redundancy-removing constraint because we consider only a limited set of operators and a limited set of nodes.
We consider the relaxation that replaces $y_{1}^{\sqrt{}}, y_{1}^{x_1}, y_{3}^{x_1} \in \{0,1\}$ with $y_{1}^{\sqrt{}}, y_{1}^{x_1}, y_{3}^{x_1} \in [0,1]$.
Consider the solution $(y_{1}^{\sqrt{}}, \allowbreak y_{1}^{x_1}, y_{3}^{x_1}, \allowbreak v_{1,1}, \allowbreak v_{1,3}) \allowbreak = (0.9, 0.1, 0.9, 0, -0.9)$.
The solution is feasible in the relaxation, whereas it violates $y_1^{\sqrt{}} + y_3^{x_1} \le 1$, which is \cref{constr:impl2}.
\end{example}

\subsubsection{Symmetry-Breaking Constraints}
\label{sec:symmetry}
Symmetry-breaking constraints remove equivalent expression trees because of a symmetric operator including addition ($+$) and multiplication ($*$), for example, $x_1 + x_2 = x_2 + x_1$.
We retain only those expression trees in which the left argument value is greater than or equal to the right argument value for the first data point:
\begin{align}
    && v_{1,2n} - v_{1,2n+1} & \ge (v_{\text{lo}} - v_{\text{up}}) ( 1 - y^+_n - y^*_n ), && n \in \mathcal{N}_{\text{perfect}}. \label{eqn:main:sym}
\end{align}
Symmetry-breaking constraints are  discussed in \cite[(5)]{cozad2018global}.
We rewrite the constraints in terms of our notations because we assume that $\mathcal{N}$ corresponds to a full binary tree whereas \cite{cozad2018global} assumes that $\mathcal{N}$ corresponds to a perfect binary tree.

\subsection{Summary} 
We summarize the formulation and our contributions in \cref{tab:formulation}.
\begin{table}[htb]
    \centering
    \caption{Summary of the formulation and contributions compared with the benchmark formulation \cite{cozad2018global}.}
    \label{tab:formulation}
    \adjustbox{max width=\textwidth}{
    \begin{tabular}{c|c|c|c|c} \hline
    Categories & Variables & Convexity & Remove additional & Reduce \\
    &&& equiv. expressions & relaxation space \\ \hline
    Objective & $v$ & convex & & \\ 
    Tree-defining constraints & $y,c$ & linear & $\checkmark$ & $\checkmark$ \\
    Value-defining constraints & $y,c,v$ & nonconvex &  & $\checkmark$\\
    Redundancy-eliminating constraints & $y$ & linear & $\checkmark$ & $\checkmark$\\
    Implication cuts & $y$ & linear & & $\checkmark$\\
    Symmetry-breaking constraints & $y,v$ & linear & & \\ \hline 
    \end{tabular}
    }
\end{table}

The formulation we propose for symbolic regression is
\begin{align*}
\min_{y,c,v} \quad & \frac{1}{n_{\text{data}}}\sum_{i=1}^{n_{\text{data}}} (z_i - v_{i,1})^2 \\
\mbox{s.t.} \quad  & \mbox{\cref{constr:grammar1}-\cref{constr:grammar2}}, \mbox{\cref{cozad:grammar1}-\cref{cozad:grammar2}}, && \mbox{(Tree-defining constraints)} \\
& \mbox{\cref{eqn:varub}-\cref{eqn:varlb}}, \mbox{\cref{eqn:cstub}-\cref{eqn:logdomain}}, && \mbox{(Value-defining constraints)} \\
& \mbox{\cref{constr:redun2}-\cref{constr:redun1}}, \mbox{\cref{cozad:redun4}-\cref{cozad:redun6}}, && \mbox{(Redundancy-eliminating constraints)} \\
& \mbox{\cref{constr:impl1}-\cref{constr:impl4}}, && \mbox{(Implication cuts)} \\
& \mbox{\cref{eqn:main:sym}}, && \mbox{(Symmetry-breaking constraints)} \\
& y_n^o \in \{0,1\}, && \forall (n,o) \in \mathcal{Y}, \\
& c_{\text{lo}} \le c_n \le c_{\text{up}}, && \forall n \in \mathcal{N}, \\
& v_{\text{lo}} \le v_{i,n} \le v_{\text{up}}, && \forall i \in [n_{\text{data}}], ~\forall n \in \mathcal{N}.
\end{align*}

\section{Sequential Tree Construction Heuristic}
\label{sec:sr-heuristic}

In this section, we propose a heuristic that searches a solution by building up an
expression tree starting from a simple approximation to a comprehensive formula.
It repeatedly modifies a part of the current expression tree in order to lower the training error.
It is motivated by the fact that a comprehensive formula can be approximated by a simple formula, and we observe that those two formulas have similar structures.
For example, $\frac{1 - x + x^2 }{1 + x}$ can be approximated by $\frac{1-x}{1+x}$ when $|x|$ is small, see \cref{fig:movitation}.
We can achieve the comprehensive formula by adding $x^2$ to the simple formula. 
Another example can be found in Kepler's third law, where the comprehensive formula $d = \sqrt[3]{c\tau^2 (M+m)}$ is approximated by a simple formula $d = \sqrt[3]{c\tau^2 M}$ for $M \gg m$.

\begin{figure}[htb]
    \centering
    \subfloat[$x \in (-1,2)$]{\includegraphics[width=0.47\textwidth]{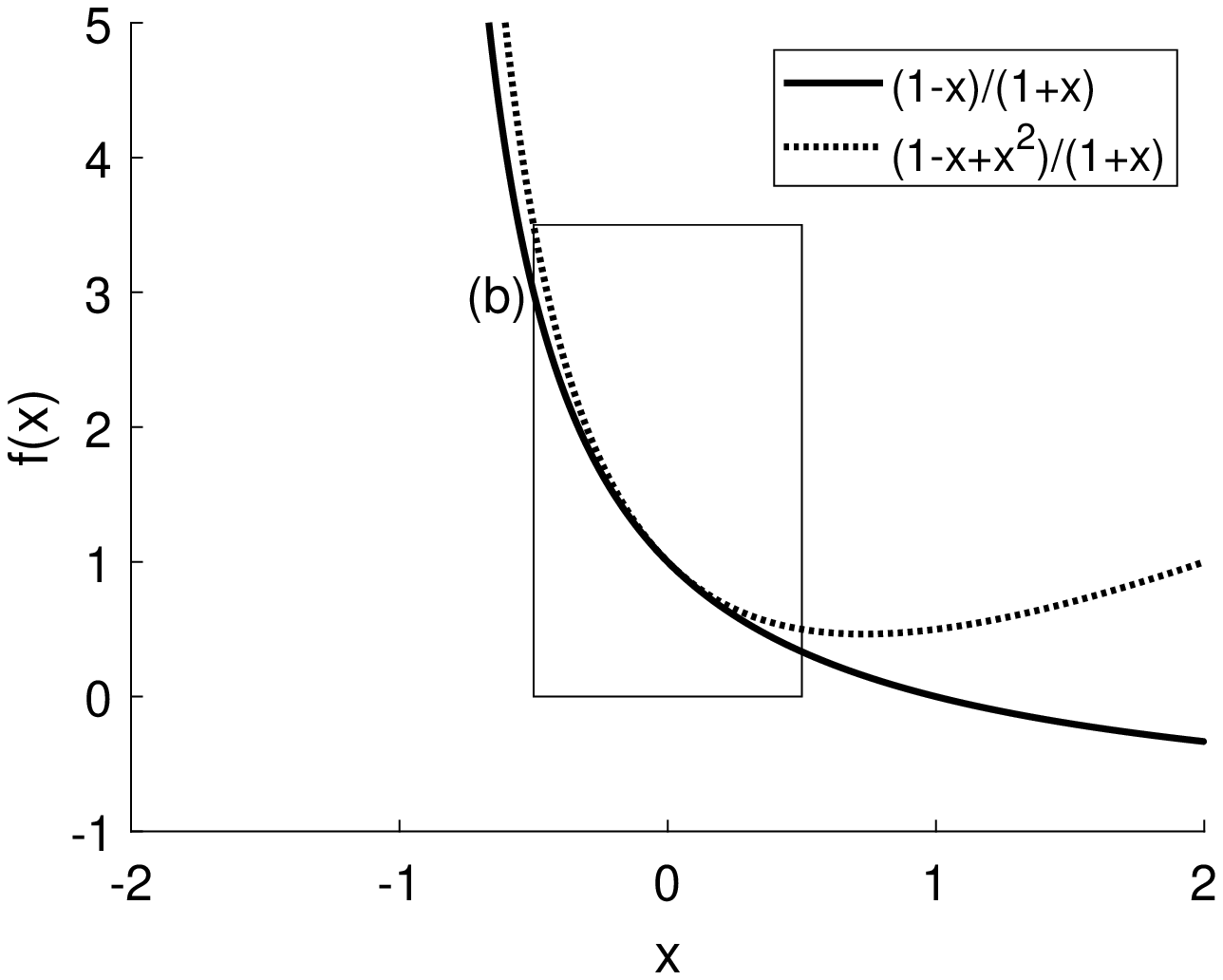}}
    \subfloat[$x \in (-0.5,0.5)$]{\includegraphics[width=0.47\textwidth]{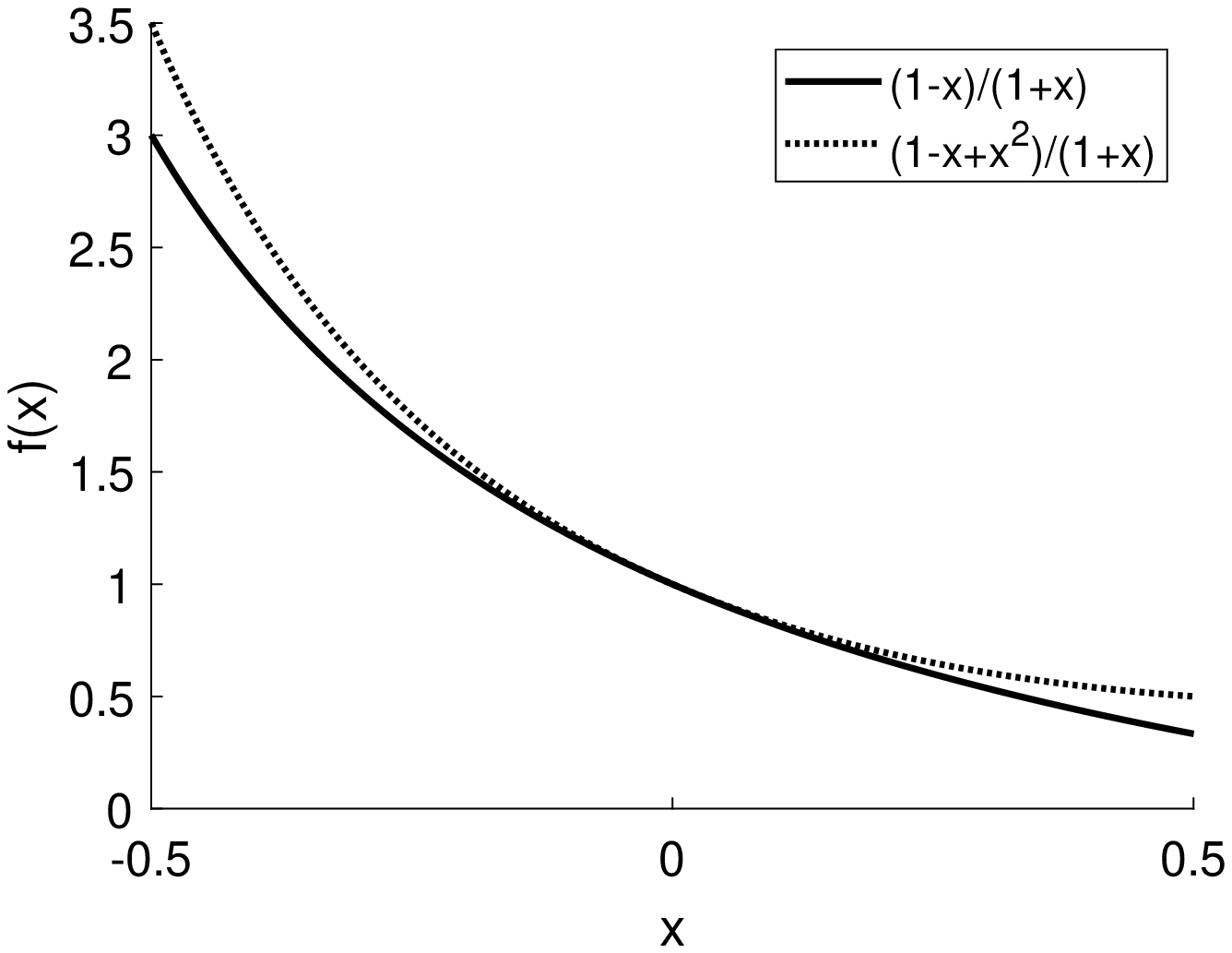}}
    \caption{Illustration of $\frac{1 - x + x^2 }{1 + x}$ and $\frac{1-x}{1+x}$.}
    \label{fig:movitation}
\end{figure}

Our approach is motivated by a heuristic framework for general MINLPs that searches an improved solution from the neighbors of the current solution, namely, local branching proposed by \cite{fischetti2003local}.
A local branching heuristic iteratively explores the neighbors by solving a restricted MINLP with a branch-and-bound solver.
We develop a \emph{sequential tree construction heuristic (STreCH)} based on the formulation in \cref{sec:sr-formulation}.
In \cref{sec:heur-basic} we define the neighbors of an expression tree that is the core of our heuristic, and in \cref{sec:heur-fast} we  discuss how to speed up the heuristic.

\subsection{Definition of Neighbors and a Basic Heuristic}  
\label{sec:heur-basic}

We define the distance between two expression trees as the number of nodes assigned different operators/operands.
There are three cases of a node with different assignment:
\begin{itemize}
    \item a change of a node assignment from an operator/operand to another operator/operand,
    \item an addition of a new node, and
    \item a deletion of a node.
\end{itemize}
We consider constant as an operand.
This distance does not count a change from a constant value to another constant value.
For example, the distance between two trees in \cref{fig:heur-neighbor} is three because of one change (from $x_1$ to $*$) and two additions of a node. 
The change in the constant value (from $2.0$ to $1.5$) is not counted.

We define a $k$-neighbor of a given solution $\bar{y}$ as an expression tree with distance at most $k$.
Let $\mathcal{N}_\text{active}(\bar{y})$ be the set of nodes in which an operator/operand is assigned.
Let $o(n,\bar{y})$ for $n \in \mathcal{N}_\text{active}(\bar{y})$ be the operator/operand that is assigned to node $n$. 
The set of $k$-neighbors of $\bar{y}$, $\mathcal{NB}_k(\bar{y})$, is represented as follows: 
\begin{align}
    \delta(\bar{y},y) & = \sum_{n \in \mathcal{N}_{\text{active}}(\bar{y})}
    (1-y_{n}^{o(n,\bar{y})}) + 
    \sum_{n \notin \mathcal{N}_{\text{active}}(\bar{y})} \sum_{o \in \mathcal{O}} y_{n}^{o}, \label{eqn:dist}\\
    \mathcal{NB}_k(\bar{y}) & = \left\{y ~\middle|~ \delta(\bar{y},y) \le k\right\}. \label{eqn:neighbor}
\end{align}
The restricted MINLP that searches $k$-neighbors needs a single additional linear constraint described in \cref{eqn:neighbor}.

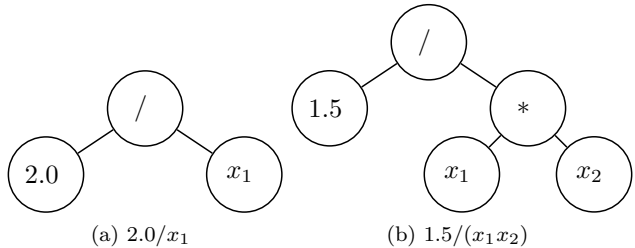
\begin{figure}[htb]
    \centering
    \subfloat[$2.0/x_1$]{
        \begin{tikzpicture}
        [scale=0.5,
        every node/.style = {shape=circle, align=center, draw, line width=0.2mm, minimum size = 10mm
        },
        edge from parent/.style = {draw, line width=0.2mm},
        level distance=5em,
        ]
        \tikzstyle{level 1}=[sibling distance=15em]
        \tikzstyle{level 2}=[sibling distance=10em]
        \tikzstyle{level 3}=[sibling distance=10em]
        \node [] { $/$ }
            child [] { node [] { 2.0 } 
                }
            child [] { node [] { $x_{1}$ } 
                };
        \end{tikzpicture}
    }
    \subfloat[$1.5/(x_1x_2)$]{
        \begin{tikzpicture}
        [scale=0.5,
        every node/.style = {shape=circle, align=center, draw, line width=0.2mm, minimum size = 10mm
        },
        edge from parent/.style = {draw, line width=0.2mm},
        level distance=5em,
        ]
        \tikzstyle{level 1}=[sibling distance=15em]
        \tikzstyle{level 2}=[sibling distance=10em]
        \tikzstyle{level 3}=[sibling distance=10em]
        \node [] { $/$ }
            child [] { node [] { 1.5 } 
                }
            child [] { node [] { $*$ } 
                child [] { node [] { $x_1$ }
                    }
                child [] { node [] { $x_2$ } 
                    } 
                };
        \end{tikzpicture}    
    }
    \caption{Two expression trees with distance three.}
    \label{fig:heur-neighbor}
\end{figure}

Our basic heuristic works as follows.
It first obtains an initial solution by solving a small-sized problem.
At each iteration, it searches for an improvement of the current solution by exploring its neighbors by solving a restricted MINLP with \cref{eqn:neighbor}.
It terminates by optimality (the objective value is less than a given tolerance $\epsilon$) or by time limit.
In our implementation we encounter the following situations:
\begin{itemize}
    \item When the tree size (the number of nodes) of the incumbent solution is large, it takes a long time to search its neighbors. 
    \item There is no better solution in its neighbor at some iteration.
\end{itemize}
We next propose some ideas to mitigate these situations.

\begin{algorithm}[tb]
\caption{\texttt{ResMINLP} (approximately solving a restricted MINLP).}
\label{algo:MINLP}
\begin{algorithm2e}[H]
\KwData{ $(x,z)$, $\mathcal{P}$, $\mathcal{N}$; (optional) $\bar{y}$, $k_1$, $k_2$, $\beta$, $\gamma$  }
\KwResult{An expression tree $(y,c)$ found by MINLP and its training error $\omega$}
Formulate a MINLP with with data $(x,z)$, operators $\mathcal{P}$, nodes $\mathcal{N}$\;
\If{$\bar{y}$, $k_1$, and $k_2$ are given}
{
    Add constraint $k_1 \le \delta(\bar{y},y) \le k_2$ to search within $\mathcal{NB}_{k_2}(\bar{y}) \setminus \mathcal{NB}_{k_1}(\bar{y})$\;
}
\If{$\bar{y}$ and $\beta$ are given}
{
    Fix some variables in $y$ by given solution $\bar{y}$ and fix level $\beta$\;
}
\If{$\gamma$ is given}
{
    Set the node limit of a branch-and-bound solver to $\gamma$\;
}
Solve the problem with a branch-and-bound solver\;
\Return ($y$, $c$, $\omega$)\;
\end{algorithm2e}
\end{algorithm}

\subsection{Heuristics to Speed Up Ideas and STreCH}
\label{sec:heur-fast}

In this section we propose  STreCH, whose pseudocode is described in \cref{algo:STreCH}.
\Cref{algo:MINLP} describes a restricted MINLP solve, and 
\Cref{algo:localsearch} describes the solution improvement procedure.
In addition, we propose a number of heuristics to speed up the solution process:

\paragraph{Early Termination} 
A branch-and-bound solver returns an optimal solution with the proof of optimality.
However, proving optimality for a restricted problem does not guarantee optimality of the whole problem.
Therefore, we terminate the solver under one of the following conditions: (1) when it finds an improved solution, or (2) it reaches a time-limit to prove optimality.
First, we stop an iteration when the solver finds a solution whose training error is smaller than $(100 * \delta) \%$ of the training error of the incumbent. 
Second, we stop an iteration when it reaches at a predetermined time limit or node limit.

\paragraph{Start Value}
We provide the incumbent as a starting point. 
In general, it is not always efficient especially when we are looking for an optimal solution with the proof of optimality.
However, it helps in combination with early termination.

\paragraph{Fix a Part of an Expression Tree}
As the size of an expression tree grows, the size of $k$-neighbors increases. We fix the top part of an expression tree to reduce the search space.
Specifically, given an incumbent and $\beta \in \mathbb{Z}_{++}$, we fix a node that has a descendant at the $\beta$-th lower generation where the first lower generation is the children and the $i$th generation is the children of $(i-1)$th generation.
For example, when $\beta = 1$, we fix all nonleaf nodes. 
When $\beta = 2$, we fix all grandparents of some node.

\begin{algorithm}[tb]
\caption{\texttt{Improve} (procedure to find an improved solution from a solution).}
\label{algo:localsearch}
\begin{algorithm2e}[H]
\KwData{ $(x,z)$, $\mathcal{P}$, $\mathcal{N}$, solution $(\bar{y}, \bar{c}, \bar{\omega})$ }
\KwResult{ An improved solution if found, otherwise the given solution}

$(k_1, k_2, \beta, \gamma) \leftarrow (0, k_{\text{init}}, \beta_{\text{init}}, \gamma_{\text{init}})$ \tcp*{Initialize the parameters.}

\Repeat(\tcp*[f]{Repeatedly search neighbors.}){Time limit reached}{
	$(y,c,\omega) \leftarrow \texttt{ResMINLP}(x,z,\mathcal{P},\mathcal{N},k_1,k_2,\beta,\gamma)$\;
	\uIf{$\omega < \bar{\omega}$}
	{
	    
		\Return $(y,c,\omega)$ \tcp*{Return an improved solution.}
	}	
	\uElseIf{Terminated by node limit and $10\gamma \le \gamma_{\max}$}{
		$\gamma \leftarrow 10\gamma$ \tcp*{Spend more time on this problem.}
	}
	\Else(\tcp*[f]{Spent enough time.})
	{
		$(k_1,k_2) \leftarrow (k_2,k_2+2)$ \tcp*{Change the neighbor set.}
		\If{$k_2 > k_{\max}$}
		{
			$(k_1, k_2) \leftarrow (0, k_{\text{init}})$\;
			$\beta \leftarrow \beta+1$\;
			\If{$\beta > \beta_{\max}$}
			{
				\Return $(\bar{y}, \bar{c}, \bar{\omega})$ \tcp*{Return the given solution.}
			}			
		}
	}
}
\end{algorithm2e}
\end{algorithm}

With the implementation of early termination, there are three situations when a MINLP solver terminates at an iteration.
Let $\bar{y}$ denote the current incumbent and $k$ denote the current distance.
We define the next iteration for each situation as follows:
\begin{enumerate}
    \item The solver returns a better solution. Then, we solve a MINLP restricted by the neighbors of the returned solution.
    \item The solver proves that there is no better solution within the neighbors.
    Then, we search in a larger neighborhood, $\mathcal{NB}_{k'}(\bar{y})\setminus\mathcal{NB}_k(\bar{y})$ where $k' > k$. 
    \item The solver terminates by time limit or node limit and no better solution has been found. 
    Then, we increase the time limit or the node limit in the next iteration.
\end{enumerate}


The pseudocode of  STreCH is described in \cref{algo:MINLP,algo:localsearch,algo:STreCH}.
\Cref{algo:MINLP} describes a restricted MINLP.
\Cref{algo:localsearch} describes the solution-improving procedure.
The inputs of \cref{algo:localsearch} are the solution $\bar{y}$ and the parameters that are given in the beginning and not changed during the procedure: $k_{\text{init}}$ is the initial distance to define neighbors, $k_{\max}$ is the maximum distance of candidate neighbors, $\beta_{\text{init}}$ is the initial node fix level, $\beta_{\max}$ is the maximum node fix level, $\gamma_{\text{init}}$ is the initial node limit for the branch-and-bound tree, and $\gamma_{\max}$ is the maximum node limit for the branch-and-bound tree.
\Cref{algo:STreCH} describes the overall loop.

We recommend  solving a single MINLP by limiting the number of nodes in the branch-and-bound tree instead of limiting time in the heuristic.
The reason is that limiting the number of node guarantees that the same solution will be reproduced at each iteration whereas limiting time may return a different solution depending on how much resource is available on the computing machine.

\begin{algorithm}[t]
\caption{STreCH.}
\label{algo:STreCH}
\begin{algorithm2e}[H]
\KwData{ $(x,z)$, $\mathcal{P}$, $\mathcal{N}$, $\mathcal{N}_{\text{init}}$ (the node set for the initial problem)  }
\KwResult{An expression tree $(y,v)$ and its training error $\omega$}
$(y,c,\omega) \leftarrow \texttt{ResMINLP}(x,z,\mathcal{P},\mathcal{N}_{\text{init}})$ \tcp*{Solve an initial problem.}

\Repeat(\tcp*[f]{Repeatedly improve a solution.}){Time limit reached}{
    $(y',c',\omega') \leftarrow \texttt{Improve}(x,z,\mathcal{P},\mathcal{N},y,c,\omega)$\;
    \uIf{$\omega' < \omega$}
    {
        
        $(y, c, \omega) \leftarrow (y',c',\omega')$ \tcp*{Update the incumbent.}
    }
    \Else
    {
        break\;   
    }
}
\Return $(y,c,\omega)$\;
\end{algorithm2e}
\end{algorithm}

\section{Computational Experiments} 
\label{sec:sr-exp}

We perform computational experiments on the improved formulation and the sequential tree construction heuristic. 
The first experiment tests our ability to find a global solution.
We investigate whether the new constraints deliver an improvement in computation.
The second experiment tests the ability to find a good approximated symbolic function with limited information.
We assume that a limited number of observations is given and no additional information such as the unit of variables is available.
We compare our methods with AI Feynman \cite{udrescu2020ai2,udrescu2020ai}, which is a state-of-the-art symbolic regression solver specialized for physics formulas.

\paragraph{Test Problems}
We test 71 formulas from the Feynman database for symbolic regression \cite{udrescu2020ai} whose operator set is a subset of $\{+,-,*,/,\sqrt{}\}$.
\Cref{tab:depth} shows that all formulas can be represented by an expression tree of depth five. 
\begin{table}[htb]
    \centering
    \begin{tabular}{c|c|c|c|c|c|c} \hline 
    Depth & 1 & 2 & 3 & 4 & 5 & Total  \\ \hline 
    \# of formulas & 4 & 25 & 22 & 7 & 13 & 71 \\ \hline
    \end{tabular}
    \caption{Distribution of the required depth to represent a formula in the test problems.}
    \label{tab:depth}
\end{table}
We assume that (i) no unit information is available, (ii) we have a small number of observations (10 data points), and (iii) the observations are noisy.
Although (i)--(iii)  were discussed in \cite{udrescu2020ai} independently, the combination of all three was not discussed.

\paragraph{Hardware and Software}

Our computational experiments are performed on a computer with Intel Xeon Gold 6130 CPU cores running at 2.10 GHz and 192 GB of memory.
The operating system is Linux Ubuntu 18.04.
The code is written in Julia 1.5.3 with SCIP 7.0.0 \cite{GleixnerEtal2018ZR} as a MINLP solver that showed the best performance among open-source global MINLP solvers for this problem \cite{cozad2018global}.
The code is available at \url{https://github.com/jongeunkim/STreCH}.

\subsection{Comparison of MINLP Formulations}

We compare the formulations described in \cref{sec:sr-formulation} with the formulation from \cite{cozad2018global}.

\subsubsection{Experimental Setup for Comparing MINLP Formulations}

We start by investigating the effect of the new optional constraints that do not need to be included in the formulation but can reduce the search space.
In the experiments in \cite{cozad2018global}, the variance of the computational performance is large with regard to the inclusion/exclusion of optional constraints, and the authors
 suggest running all possible formulations in parallel.
In this experiment, we consider four formulations for each method (ours and \cite{cozad2018global}) by adding or not adding redundancy-eliminating constraints and symmetry-breaking constraints.
\emph{Imp} and \emph{Coz} stand for the improved formulation and the formulation from \cite{cozad2018global}, respectively.
\emph{-F} refers to the full formulation (adding all the constraints).
\emph{-N} refers to the formulation  with only the necessary constraints (tree and value defining constraints).
\emph{-R} refers to the formulation with the necessary constraints and the redundancy-eliminating constraints.
\emph{-S} refers to the formulation with the necessary constraints and the symmetry-breaking constraints.
The configuration of the formulations are described in \cref{tab:formulations}.
\begin{table}[htb]
    \centering
    \adjustbox{max width=\textwidth}{
    \begin{tabular}{|c||P{20mm}|P{20mm}|P{20mm}|P{20mm}|P{20mm}|P{20mm}|P{20mm}|P{20mm}|} \hline
    Formulations & Imp-F & Imp-R & Imp-S & Imp-N & Coz-F & Coz-R & Coz-S & Coz-N \\
    \hline \hline
    Objective & \mc{8}{c|}{\cref{obj:main}} \\ \hline
    Tree & \mc{4}{c|}{\cref{constr:grammar1}-\cref{constr:grammar2}, \cref{cozad:grammar1}-\cref{cozad:grammar2}} & \mc{4}{c|}{\cref{cozad:grammar1}-\cref{cozad:grammar6}} \\ \hline
    Value & \mc{4}{c|}{\cref{eqn:varub}-\cref{eqn:varlb}, \cref{eqn:cstub}-\cref{eqn:logdomain}} & \mc{4}{c|}{\cref{eqn:noneub}-\cref{eqn:logdomain}} \\ \hline
    Redundancy & \mc{2}{c|}{\cref{constr:redun2}-\cref{constr:redun1}, \cref{cozad:redun4}-\cref{cozad:redun6}} & \mc{2}{c|}{-} & \mc{2}{c|}{\cref{cozad:redun1}-\cref{cozad:redun6}} & \mc{2}{c|}{-} \\ \hline
    Symmetry & \cref{eqn:main:sym} & - & \cref{eqn:main:sym} & - & \cref{eqn:main:sym} & - & \cref{eqn:main:sym} & - \\  \hline
    \end{tabular}
    }
    \caption{Formulations used in the experiments.}
    \label{tab:formulations}
\end{table}
We do not consider implication cuts because all the independent variables used in the test functions are positive, which means that there are no implication cuts.
We limit the depth of the expression tree to two (seven nodes) in order to find an optimal solution for all instances within the prespecified time limit (three hours).

\subsubsection{Results Comparing MINLP Formulations}

First, we compare our best results with the best results of \cite{cozad2018global} in \cref{fig:formulation_time_best}.
We collect the smallest solution times among four formulations for each method.
We visualize our results using performance profiles \cite{dolan2002benchmarking} in \cref{fig:formulation}.
\Cref{fig:formulation_time_best} shows that our formulations can solve 70\% of instances within ten minutes while \cite{cozad2018global}'s can solve 50\% of instances.
Our formulations failed to solve 2.8\% of instances (2 of 71) within three hours while \cite{cozad2018global}'s  failed to solve 5.6\% of instances (4 of 71) within the time
limit.

\begin{figure}[htbp]
	\subfloat[Best of four.]{
		\label{fig:formulation_time_best}
		\includegraphics[width=0.47\textwidth]{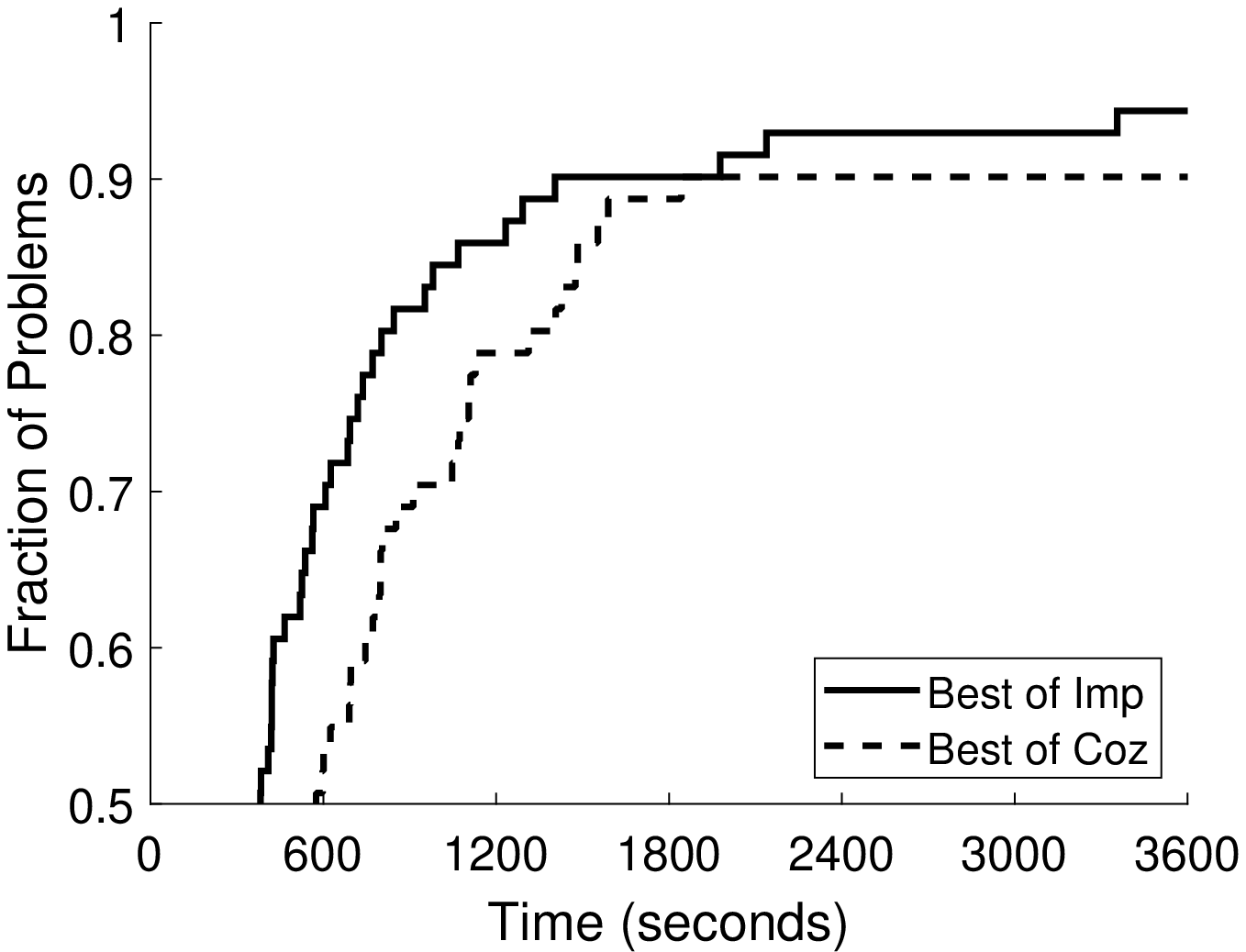}
		}
	\subfloat[All.]{
		\label{fig:formulation_time_all}
		\includegraphics[width=0.47\textwidth]{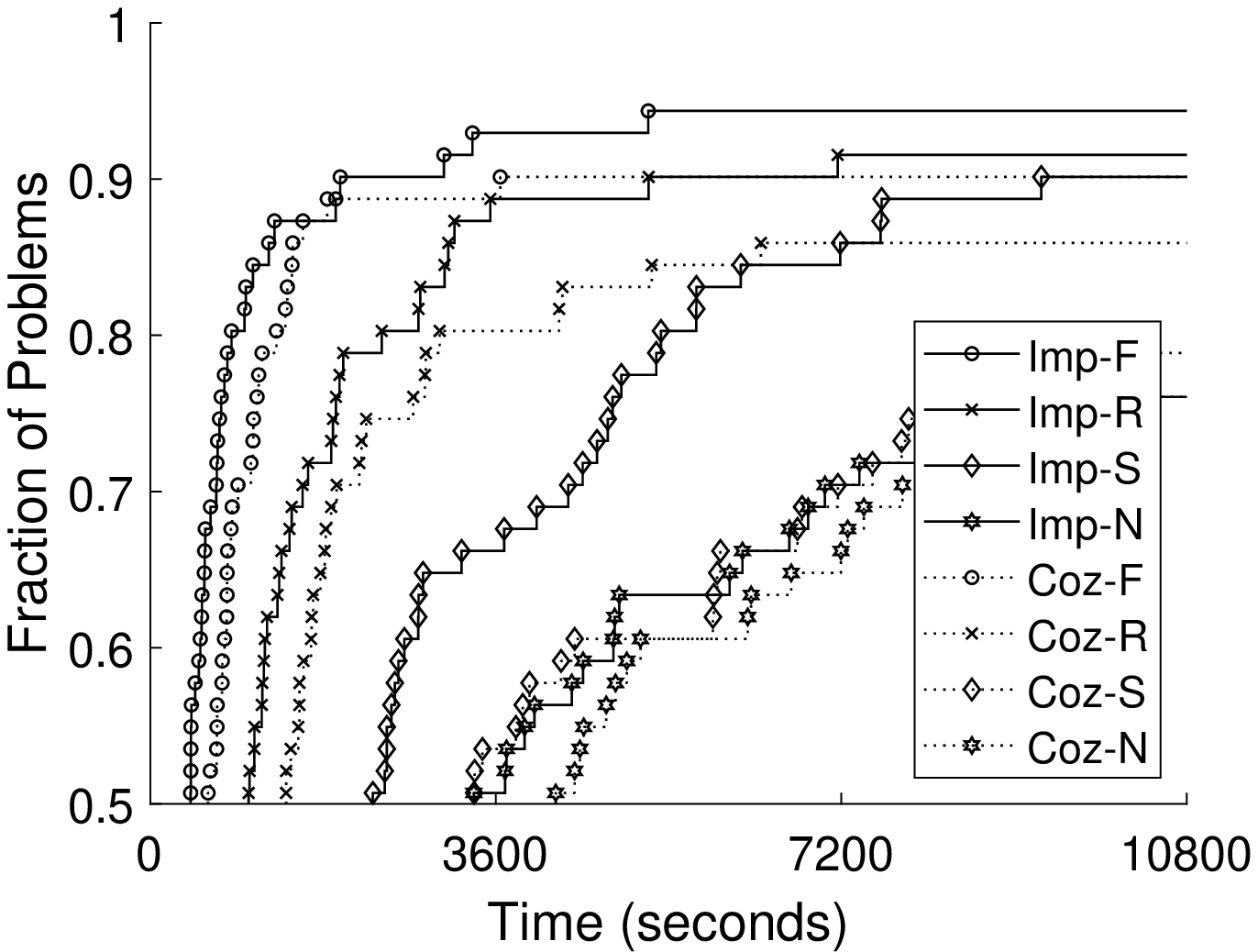}
		}
	\caption{Comparison of solution times of MINLP solves with the improved formulations and those of \cite{cozad2018global}.}
	\label{fig:formulation}
\end{figure}

Next, we compare all eight formulations in \cref{fig:formulation_time_all}.
Now we draw eight lines for each formulation.
\Cref{fig:formulation_time_all} shows that our formulation with all optional constraints performs best.
The formulation terminates first in more than  half of the instances (36 of 71) and is in the top three in 81.7\% of the instances (58 of 71).
\Cref{fig:formulation-bnbnodes} shows that the newly proposed optional constraints in the improved formulation also reduce the number of nodes in the branch-and-bound tree (BnBnodes) compared with  the existing ones.

Our experiments show that it clearly  is better to add all optional constraints to reduce the search space. 
This conclusion differs from the result in \cite{cozad2018global}. We believe that this 
difference may be due to a difference in the branch-and-bound solvers: 
we use SCIP while \cite{cozad2018global} uses BARON \cite{sahinidis:baron:17.8.9,ts:05}. 

\subsection{STreCH and AI Feynman}
\label{sec:exp:heur}

The goal of this experiment is to compare the performance of our methods with the state-of-the-art symbolic regression method AI Feynman. 

\subsubsection{Experimental Setup for Heuristic Comparison}

We test both methods with noisy data and perform cross-validation to select the final symbolic expression. 
We first generate a training set with noise and a validation and testing set without noise.
For each method, we find a set of candidate formulas using the training set.
Next, we select the formula from the candidates that has the lowest validation error.
We then compute the testing error of the selected formula.

We consider three approaches: a STreCH-based approach, a MINLP-based approach, and AI Feynman.
The first two approaches solve multiple instances with different parameter setups in parallel to collect candidate formulas.
The set of parameters includes the type of formulation (adding or not adding optional constraints), the maximum depth of the expression tree, the type of constant (integer or fractional), and the bounds on the constants.
When a single instance is solved, the STreCH-based approach uses the heuristic described in \cref{sec:sr-heuristic}, and the MINLP-based approach solves the MINLP problem in \cref{sec:sr-formulation}.


We run AI Feynman ourselves because there are no reported computational experiments in our setting (running without no unit information, for a small number of observations, and noisy data).
We download the AI Feynman code from \url{https://github.com/SJ001/AI-Feynman}. 
We use the default parameters except the set of operators used in brute-force because the default set does not include all operators used in tested functions.
Instead, we use the largest operator set that includes all used operators.
Note that the AI Feynman code itself manages computing resources in parallel.

\begin{figure}[tb]
    \centering
	\includegraphics[width=0.7\textwidth]{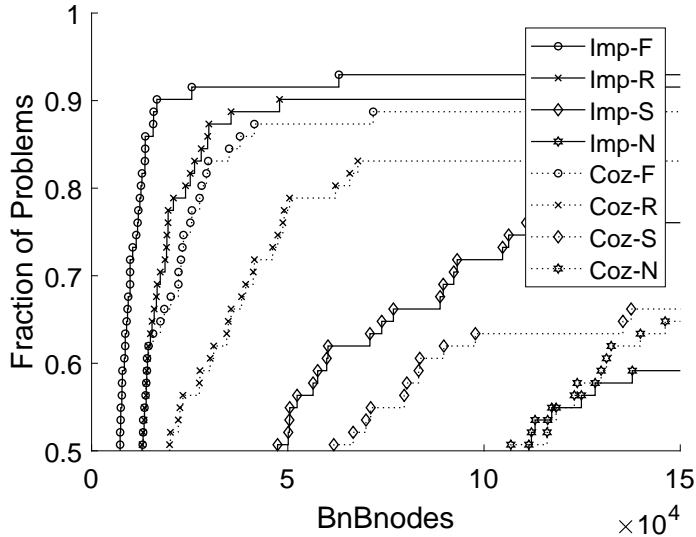} 
	\caption{Comparison of the number of nodes in the branch-and-bound tree of MINLP solves with the improved formulations and the the formulations of \cite{cozad2018global}.}
	\label{fig:formulation-bnbnodes}
\end{figure}

\subsubsection{Results for Comparison of Heuristics}

First, we investigate how many formulas can be rediscovered by each method.
We run all methods on the dataset with ten training data points, a noise level of $10^{-4}$, and no unit information.
In \cref{tab:rediscover}, we observe that every method can discover the correct formula when it can be represented by an expression tree of depth one or two.
When the depth is three, the STreCH and the MINLP approaches discover twice as many formulas as AI Feynman.
When the depth is four or five, the STreCH and the MINLP approaches cannot discover the original formulas, while AI Feynman discovers two formulas.

\begin{table}[htb]
    \centering
    \caption{Required depth to represent a formula.}
    \label{tab:rediscover}
    \begin{tabular}{r|r|r|r|r} \hline 
    Depth & \# Formulas & \mc{3}{c}{Discovery rate (\%)} \\ \cline{3-5}
    & & STreCH & MINLP & AI Feynman  \\ \hline 
    $\le 2$ & 29 & 100.0 & 100.0 & 100.0 \\ 
   	3 & 22 & 59.1 & 54.5 & 27.2 \\
   	$\ge 4$ & 20 & 0.0 & 0.0 & 10.0 \\ 
   	\hline
    \end{tabular}
\end{table}

Next, we investigate formulas for which the methods return different solutions.
We first consider formulas that were discovered by the STreCH and the MINLP approaches but not by AI Feynman.
These include $\frac{q_1q_2r}{4\pi\epsilon r^3}$ (Feynman Eq. I.12.2) and $\frac{2I}{4\pi\epsilon c^2 r}$ (Feynman Eq. II.13.17).
We suspect that this difference happens because the STreCH and the MINLP approaches can assign any constant value at a node in the expression tree whereas AI Feynman relies on the user-specified particular constants. 
Specifically, when solving the problem of $\frac{q_1q_2r}{4\pi\epsilon r^3}$, the STreCH and the MINLP approaches can assign $0.159 (=\frac{1}{2\pi})$ at a node while AI Feynman needs a few steps in combination with a prespecified constant ($\pi$) and operators ($x \rightarrow 2x$ and $x \rightarrow \frac{1}{x}$).
These few steps might hinder the ability of the path to rediscover the original formula. 
This weakness also has been pointed out in \cite{austel2020symbolic}.

Second, we investigate formulas that were discovered by AI Feynman but not by the STreCH and the MINLP approaches.
These include $\frac{m_0}{\sqrt{1-v^2/c^2}}$ and $\frac{\rho_{c_0}}{\sqrt{1-v^2/c^2}}$.
AI Feynman benefits from the use of trigonometrical functions, $\arcsin(\cos(x))$, which are equivalent to $\sqrt{1-x^2}$.
Third, we consider formulas where the STreCH approach could find the original formula but the MINLP approach failed.
These include $\frac{(h/(2\pi))^2}{2E_nd^2}$ (Feynman Eq. III.15.14).
\Cref{tab:heuristic} shows how the STreCH develops a solution at each iteration.
Because the formula is a monomial, there are multiple sequences to reach the correct formula. For example, the correct formula can be obtained from $\frac{h^2}{E_nd^2}$, $\frac{Ch^2}{d^2}$, and $\frac{Ch}{E_nd^2}$, where $C = (8\pi^2)^{-1}$.
Therefore, there are formula structures such as a monomial that the STreCH performs well. 
\begin{table}[htb]
    \centering
    \caption{Progress of STreCH to discover $\frac{(h/(2\pi))^2}{2E_nd^2}$.}
    \label{tab:heuristic}
    \begin{tabular}{c|c|c|r} \hline 
    Iteration & Incumbent$^*$ & Update & Time Spent (s)\\ \hline
	1 & $\frac{c_1h}{d}$ & initial solution$^{**}$ & 75.33 \\
	2 & $\frac{c_2h}{E_nd}$ &  $h \rightarrow h/E_n$ & 2.82 \\
	3 & $\frac{c_3h}{E_nd^2}$ & $d \rightarrow d^2$ & 68.39\\
	4 & $\frac{c_4h}{E_nd^2}$ & change the constant value & 1.73 \\
	5 & $\frac{c_5h^2}{E_nd^2}$ & $h \rightarrow h^2$ & 26.24\\
	6 & $\frac{c_6h^2}{E_nd^2}$ & change the constant value & 161.77 \\
    \hline
    \mc{4}{l}{$^*$ $c_1$-$c_6$ are constant values} \\
    \mc{4}{l}{$^{**}$ achieved by solving a depth-two problem} 
    \end{tabular}
\end{table}

\begin{figure}[tb]
	\subfloat[Noiseless.]{
		\label{fig:tsterr_noiseless}
		\includegraphics[width=0.47\textwidth]{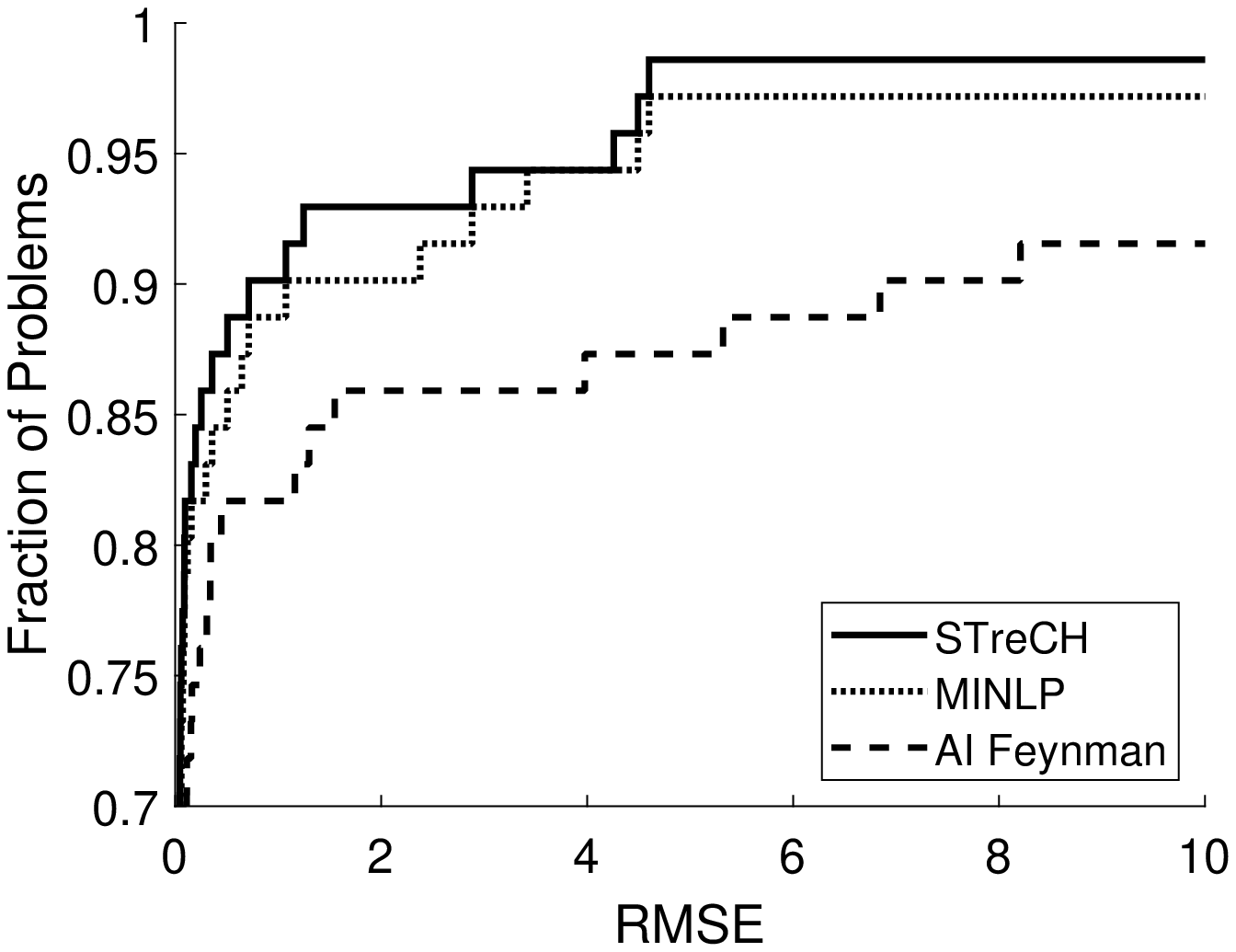}}
	\subfloat[Noise level of 0.01\%.]{
		\label{fig:tsterr_noise}
		\includegraphics[width=0.47\textwidth]{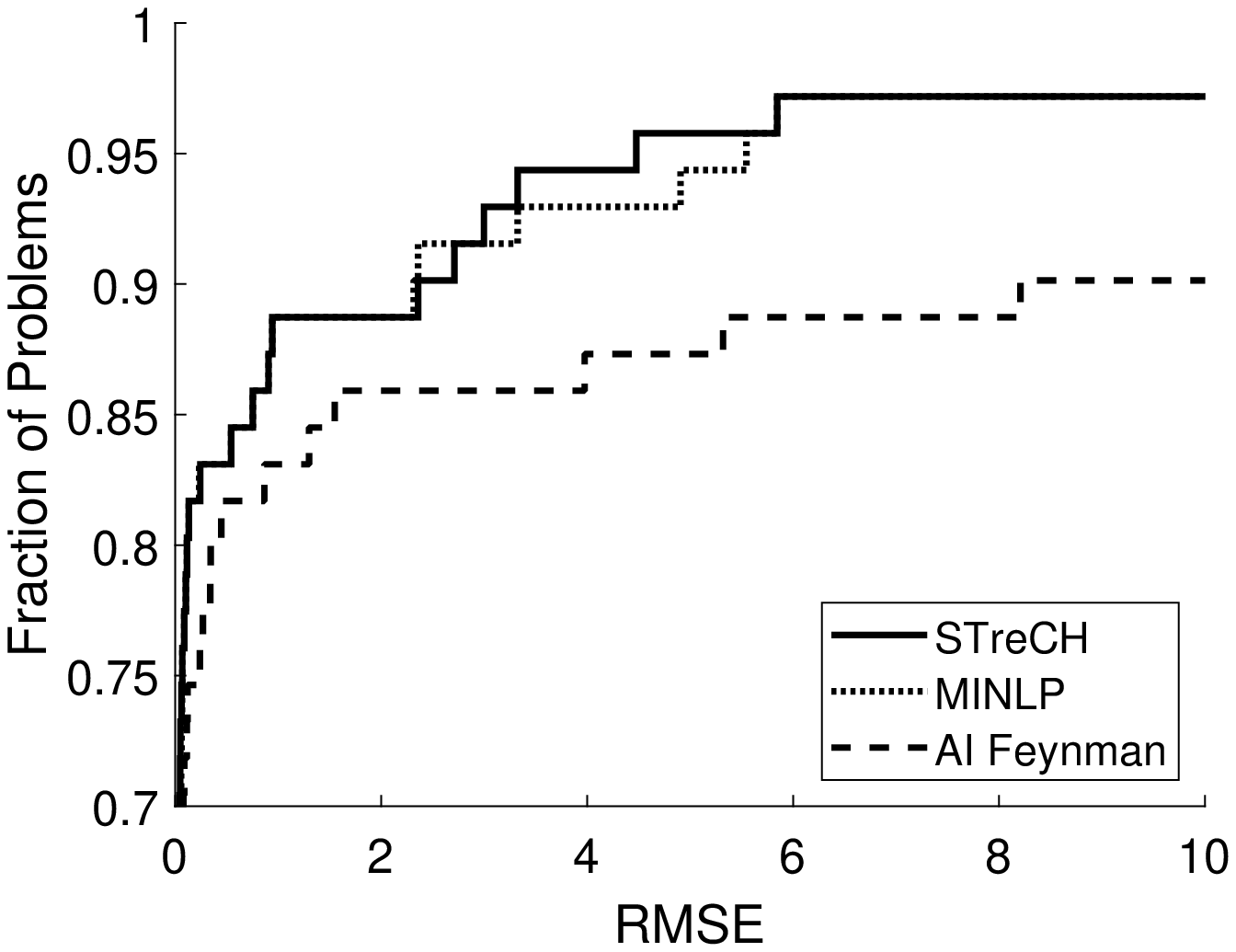}
		}
	\caption{Comparison of the testing errors achieved by the STreCH approach, the MINLP approach, and AI Feynman.}
	\label{fig:tsterr}
\end{figure}

We next compare the testing errors of the three methods.
We perform the computational experiments on both noiseless and noisy data.
\Cref{fig:tsterr} shows the distributions of the root mean square testing errors of the results.
We see that the solutions generated by the STreCH and the MINLP approaches have a lower testing error compared with AI Feynman's solutions.


%


\section{Conclusion}
\label{sec:sr-conclusion}

In this paper we present MINLP-based methods for symbolic regression.
We propose an improved MINLP formulation.
We also propose a new heuristic, STreCH, which is based on the tighter formulation and builds an expression tree by repeatedly modifying a solution expression tree.
Compared with state-of-the-art methods, our methods are able to discover more correct formulas when there is a lack of data.
When an original formula is difficult to rediscover, our methods return a formula that has a lower testing error.

For future work, our method can be integrated with AI Feynman.
AI Feynman decomposes to small problems and  solves those using polynomial fit and brute-force methods. 
Since  our method is good at finding a relatively simple symbolic expression, it would be a good option for solving small subproblems that arise within the AI Feynman decomposition within a tight time limit, of, say, less that a minute.

\subsubsection*{Acknowledgements}
This work was supported by the Applied Mathematics activity within the U.S. Department of Energy, Office of Science, Advanced Scientific Computing Research, under Contract DE-AC02-06CH11357.

\appendix

\section{The MINLP Formulation Proposed in \cite{cozad2018global}}
\label{sec:cozad-formulation}

The formulation proposed by \cite{cozad2018global} is written in terms of our notations.

\subsection{Tree-Defining Constraints}
\label{sec:cozad-treeconstr}

The constraints are (1d)--(1i) in \cite{cozad2018global}.

\begin{subequations}
\begin{align}
    \sum_{o \in \mathcal{O}} y^o_n & \le 1, && n \in \mathcal{N}, \label{cozad:grammar1} \\
    \sum_{n \in \mathcal{N}} \sum_{j=1}^d y^{x_{j}}_n & \ge 1, \label{cozad:grammar2} \\
    \sum_{o \in \mathcal{B} \cup \mathcal{U}} y^o_n & \le \sum_{o \in \mathcal{O}} y^o_{2n+1}, && n \notin \mathcal{T}, \label{cozad:grammar3} \\
    \sum_{o \in \mathcal{B}} y^o_n & \le \sum_{o \in \mathcal{O}} y^o_{2n}, && n \notin \mathcal{T}, \label{cozad:grammar4} \\ 
    \sum_{o \in \mathcal{U} \cup \mathcal{L}} y^o_n & \le 1 - \sum_{o \in \mathcal{O}} y^o_{2n}, && n \notin \mathcal{T}, \label{cozad:grammar5} \\
    \sum_{o \in \mathcal{L}} y^o_n & \le 1 - \sum_{o \in \mathcal{O}} y^o_{2n+1}, && n \notin \mathcal{T}. \label{cozad:grammar6}
\end{align}
\end{subequations}


\subsection{Value-Defining Constraints}
\label{sec:form-valueconstr}


The constraints are (1b) and (1c) in \cite{cozad2018global}.
\paragraph{No Assignment}

\begin{subequations}
\begin{align}
    v_{i,n}  & \le v_{\text{up}} (1 - \sum_{o \in \mathcal{O}} y^o_n), && \forall i \in [n_{\text{data}}], ~\forall n \in \mathcal{N}, \label{eqn:noneub} \\
    v_{i,n}  & \ge v_{\text{lo}} (1 - \sum_{o \in \mathcal{O}} y^o_n), && \forall i \in [n_{\text{data}}], ~\forall n \in \mathcal{N}. \label{eqn:nonelb}
\end{align}
\end{subequations}

\paragraph{Independent Variables}
\begin{subequations}
\begin{align}
    v_{i,n}  & \le x_{i,j}y^{x_j}_n + v_{\text{up}}(1-y^{x_j}_n), && \forall i \in [n_{\text{data}}], ~\forall n \in \mathcal{N}, ~\forall j \in [d], \label{eqn:indepub} \\
    v_{i,n}  & \ge x_{i,j}y^{x_j}_n + v_{\text{lo}}(1-y^{x_j}_n), && \forall i \in [n_{\text{data}}], ~\forall n \in \mathcal{N}, ~\forall j \in [d]. \label{eqn:indeplb}
\end{align}
\end{subequations}

\paragraph{Constant}
\begin{subequations}
\begin{align}
    v_{i,n}  - c_n & \le (v_{\text{up}} - c_{\text{lo}}) ( 1 - y^{\text{cst}}_n ), && \forall i \in [n_{\text{data}}], ~\forall n \in \mathcal{N}, \label{eqn:cstub} \\
    v_{i,n}  - c_n & \ge (v_{\text{lo}} - c_{\text{up}}) ( 1 - y^{\text{cst}}_n ), && \forall i \in [n_{\text{data}}], ~\forall n \in \mathcal{N}. \label{eqn:cstlb}
\end{align}
\end{subequations}

\paragraph{Addition}
\begin{subequations}
\begin{align}
    v_{i,n} - (v_{i,2n} + v_{i,2n+1}) & \le (v_{\text{up}} - 2v_{\text{lo}}) (1-y_n^+), && \forall i \in [n_{\text{data}}], ~\forall n \notin \mathcal{T}, \label{eqn:+ub}\\
    v_{i,n} - (v_{i,2n} + v_{i,2n+1}) & \ge (v_{\text{lo}} - 2v_{\text{up}}) (1-y_n^+), && \forall i \in [n_{\text{data}}], ~\forall n \notin \mathcal{T}. \label{eqn:+lb}
\end{align}
\end{subequations}

\paragraph{Subtraction}
\begin{subequations}
\begin{align}
    v_{i,n} - (v_{i,2n} - v_{i,2n+1}) & \le (2v_{\text{up}} - v_{\text{lo}}) (1-y_n^-), && \forall i \in [n_{\text{data}}], ~\forall n \notin \mathcal{T}, \label{eqn:-ub}\\
    v_{i,n} - (v_{i,2n} - v_{i,2n+1}) & \ge (2v_{\text{lo}} - v_{\text{up}}) (1-y_n^-), && \forall i \in [n_{\text{data}}], ~\forall n \notin \mathcal{T}. \label{eqn:-lb}
\end{align}
\end{subequations}

\paragraph{Multiplication}
\begin{subequations}
\begin{align}
    v_{i,n} - v_{i,2n}v_{i,2n+1} & \le (v_{\text{up}} - \min\{v_{\text{lo}}^2,v_{\text{lo}}v_{\text{up}},v_{\text{up}}^2\} ) (1-y_n^*), && \forall i \in [n_{\text{data}}], ~\forall n \notin \mathcal{T}, \label{eqn:*ub}\\
    v_{i,n} - v_{i,2n}v_{i,2n+1} & \ge (v_{\text{lo}} - \max\{v_{\text{lo}}^2,v_{\text{up}}^2 \}) (1-y_n^*), && \forall i \in [n_{\text{data}}], ~\forall n \notin \mathcal{T}. \label{eqn:*lb}
\end{align}
\end{subequations}

\paragraph{Division}
\begin{subequations}
\begin{align}
    v_{i,n}v_{i,2n+1} - v_{i,2n} & \le (\max\{v_{\text{lo}}^2,v_{\text{up}}^2 \} - v_{\text{lo}}) (1-y_n^/), && \forall i \in [n_{\text{data}}], ~\forall n \notin \mathcal{T}, \label{eqn:/ub}\\
    v_{i,n}v_{i,2n+1} - v_{i,2n} & \ge (\min\{v_{\text{lo}}^2,v_{\text{lo}}v_{\text{up}},v_{\text{up}}^2\} - v_{\text{up}}) (1-y_n^/), && \forall i \in [n_{\text{data}}], ~\forall n \notin \mathcal{T}, \label{eqn:/lb} \\ 
    \epsilon y_n^{/} & \le v_{i,2n}^2,  && \forall i \in [n_{\text{data}}], ~\forall n \notin \mathcal{T}, \label{eqn:/domain_lch} \\
    \epsilon y_n^{/} & \le v_{i,2n+1}^2,  && \forall i \in [n_{\text{data}}], ~\forall n \notin \mathcal{T}. \label{eqn:/domain_rch}
\end{align}
\end{subequations}

\paragraph{Square Root}
\begin{subequations}
\begin{align}
    v_{i,n}^2 - v_{i,2n+1} & \le (\max\{v_{\text{lo}}^2,v_{\text{up}}^2\} - v_{\text{lo}}) (1-y_n^{\sqrt{}}), && \forall i \in [n_{\text{data}}], ~\forall n \notin \mathcal{T}, \label{eqn:sqrtub}\\
    v_{i,n}^2 - v_{i,2n+1} & \ge (- v_{\text{up}}) (1-y_n^{\sqrt{}}), && \forall i \in [n_{\text{data}}], ~\forall n \notin \mathcal{T}, \label{eqn:sqrtlb} \\
    \epsilon - v_{i,2n+1} & \le (\epsilon - v_{\text{lo}})(1-y_n^{\sqrt{}}),  && \forall i \in [n_{\text{data}}], ~\forall n \notin \mathcal{T}. \label{eqn:sqrtdomain}
\end{align}
\end{subequations}

\paragraph{Exponential}
\begin{subequations}
\begin{align}
    v_{i,n} - \exp(v_{i,2n+1}) & \le v_{\text{up}} (1-y_n^{\exp}), && \forall i \in [n_{\text{data}}], ~\forall n \notin \mathcal{T}, \label{eqn:expub}\\
    v_{i,n} - \exp(v_{i,2n+1}) & \ge (v_{\text{lo}} - \exp(v_{\text{up}})) (1-y_n^{\exp}), && \forall i \in [n_{\text{data}}], ~\forall n \notin \mathcal{T}. \label{eqn:explb}
\end{align}
\end{subequations}

\paragraph{Logarithm}
\begin{subequations}
\begin{align}
    \exp(v_{i,n}) - v_{i,2n+1} & \le (\exp(v_{\text{up}} - v_{\text{lo}}) (1-y_n^{\log}), && \forall i \in [n_{\text{data}}], ~\forall n \notin \mathcal{T}, \label{eqn:logub} \\
    \exp(v_{i,n}) - v_{i,2n+1} & \ge (- v_{\text{up}}) (1-y_n^{\log}), && \forall i \in [n_{\text{data}}], ~\forall n \notin \mathcal{T}, \label{eqn:loglb} \\
    \epsilon - v_{i,2n+1} & \le (\epsilon - v_{\text{lo}})(1-y_n^{\log}),  && \forall i \in [n_{\text{data}}], ~\forall n \notin \mathcal{T}. \label{eqn:logdomain}
\end{align}
\end{subequations}


\subsection{Redundancy-Eliminating Constraints}
\label{sec:form-redun}

These constraints are (2a), (2b), (3a), (3b), (4a), and (4b) in \cite{cozad2018global}.
Define $\mathcal{O}_{\text{pair}}$ as the set of all inverse unary operation pairs $o$ and $o'$ such as $(\exp, \log)$, and $((\cdot)^2, \sqrt{})$.
\begin{subequations}
\begin{align}
    y^{\text{cst}}_{2n+1} + \sum_{o \in \mathcal{U}} y^o_n & \le 1, && n \notin \mathcal{T}, \label{cozad:redun1} \\
    y^{\text{cst}}_{2n+1} + y^-_n & \le 1, && n \notin \mathcal{T}, \label{cozad:redun2} \\
    y^{\text{cst}}_{2n+1} + y^/_n & \le 1, && n \notin \mathcal{T}, \label{cozad:redun3} \\
    y^{\text{cst}}_{2n} + y^{\text{cst}}_{2n+1} & \le 1, && n \notin \mathcal{T}, \label{cozad:redun4}\\ 
    y^{o}_{n} + y^{o'}_{2n+1} & \le 1, && n \notin \mathcal{T}, ~(o,o') \in \mathcal{O}_{\text{pair}}, \label{cozad:redun5}\\
    y^{o'}_{n} + y^{o}_{2n+1} & \le 1, && n \notin \mathcal{T}, ~(o,o') \in \mathcal{O}_{\text{pair}}. \label{cozad:redun6}
\end{align}
\end{subequations}


\subsection{Symmetry-Breaking Constraints}
\label{sec:form-sym}

This constraint is (5) in \cite{cozad2018global}.
\begin{align}
    v_{1,2n} - v_{1,2n+1} & \ge (v_{\text{lo}} - v_{\text{up}}) ( 1 - y^+_n - y^*_n ), && n \in \mathcal{N}_{\text{perfect}} \label{cozad:sym}
\end{align}



\section{Proof of \cref{lem:multichoice}}
\label{lem:multichoice:proof}

\begin{proof}
We first show \cref{lem:multichoice1}.
Note that $\mathcal{F}_B = \{e_i, ~\forall i \in [m]\},$ where $e_i$ is the $i$th standard unit vector where the $i$th element is one, and  otherwise zero.
Equality~\cref{lem:multichoice1} can be shown by \begin{multline*}
    \{(x,y) \in S ~|~ y \in \mathcal{F}_B\} = \bigcup_{i=1}^m \left(S \cap \{(x,y) ~|~ y=e_i\}\right)\\
     = \bigcup_{i=1}^m \left(T \cap \{(x,y) ~|~ y=e_i\}\right) = \{(x,y) \in T ~|~ y \in \mathcal{F}_B\}.
\end{multline*}
 The equality in the middle holds because of the following observations:
\begin{align*}
S \cap \{(x,y) ~|~ y=e_i\} & = \begin{cases}
\{(x,e_i) ~|~ f(x) \le w_i\} & \mbox{if } i \in [k],\\
\{(x,e_i) ~|~ f(x) \le M\} & \mbox{otherwise},
\end{cases} \\
T \cap \{(x,y) ~|~ y=e_i\} & = \begin{cases}
\{(x,e_i) ~|~ f(x) \le \min\{w_i,M\} = w_i\} & \mbox{if } i \in [k],\\
\{(x,e_i) ~|~ f(x) \le M\} & \mbox{otherwise}.
\end{cases}
\end{align*}

We next show \cref{lem:multichoice2}.
Let $S_C := \{(x,y) \in S ~|~ y \in \mathcal{F}_C\}$ and $T_C := \{(x,y) \in T ~|~ y \in \mathcal{F}_C\}$.
First, it holds that $S_C \subseteq T_C$ because the constraint in $S$ dominates every constraint in $T$.
It is sufficient to show that
\begin{multline*}
\sum_{i \in [k]} w_iy_i + M(1 - \sum_{i \in [k]} y_i)
= w_jy_j + \sum_{i \in [k]: i \neq j} w_iy_i + M(1 - \sum_{i \in [k]} y_i)\\
\le w_jy_j + \sum_{i \in [k]: i \neq j} My_i + M(1 - \sum_{i \in [k]} y_i)
= w_jy_j + M(1-y_j),
\end{multline*}
for all $j \in [k]$.
Second, we show that there exists $(x,y) \in T_C\setminus S_C$.
Let $j = \arg\min_{i\in[k]} w_i$.
Consider $(\bar{x},\bar{y})$ with $\bar{x} = \frac{1}{m}w_j+\frac{m-1}{m}M$ and $\bar{y}_i = \frac{1}{m}$ for all $i$.
Point $(\bar{x},\bar{y})$ is in $\mathcal{F}_C$. 
Point $(\bar{x},\bar{y})$ is not in $S$ because
\begin{multline*}
\bar{x} - \left(\sum_{i \in [k]} w_i\bar{y}_i + M(1 - \sum_{i \in [k]} \bar{y}_i)\right)\\
= \left(\frac{1}{m}w_j+\frac{m-1}{m}M\right) - \left(\frac{\sum_{i \in [k]} w_i}{m} + \frac{m-k}{m}M\right) =
\frac{\sum_{i \in [k]:i\neq j} (M - w_i)}{m} > 0,
\end{multline*}
while the point is in $T$ because 
\begin{multline*}
     \bar{x} - \left( w_i\bar{y}_i + M(1-\bar{y}_i)\right) = \left(\frac{1}{m}w_j+\frac{m-1}{m}M\right) - \left( \frac{1}{m}w_i + \frac{m-1}{m}M\right) \\
     = \frac{1}{m}(w_j - w_i) \le 0
\end{multline*}
for all $i \in [k]$.
Therefore, $(\bar{x},\bar{y}) \in T_C \setminus S_C$, which completes the proof.
\end{proof}
%
%

%
%




\bibliographystyle{siamplain}
\bibliography{references}

\vfill
\noindent
\fbox{\parbox{0.98\textwidth}{\small The submitted manuscript has been created by UChicago Argonne, LLC, Operator of Argonne National Laboratory ("Argonne”). Argonne, a U.S. Department of Energy Office of Science laboratory, is operated under Contract No. DE-AC02-06CH11357. The U.S. Government retains for itself, and others acting on its behalf, a paid-up nonexclusive, irrevocable worldwide license in said article to reproduce, prepare derivative works, distribute copies to the public, and perform publicly and display publicly, by or on behalf of the Government. The Department of Energy will provide public access to these results of federally sponsored research in accordance with the DOE Public Access Plan (\url{http://energy.gov/downloads/doe-public-access-plan}).}}

\end{document}